\documentclass[12pt,oneside]{amsart}

\usepackage{geometry}   % See geometry.pdf to learn the layout options. There are lots.
\usepackage{graphicx}
\usepackage{fullpage}
\usepackage{amsthm,amssymb,amsmath}
\usepackage{subfigure}
\usepackage{color}
\allowdisplaybreaks[3]

\newtheorem{theorem}{Theorem}[section]
\newtheorem{lemma}[theorem]{Lemma}

\newtheorem{corollary}[theorem]{Corollary}

\newtheorem{conjecture}[theorem]{Conjecture}

\def\ng{\widetilde{\mathsf{g}}}
\def\g{\mathsf{g}}
\def\P{\mathbb{P}}
\def\E{\mathbb{E}}
\def\F{\mathcal{F}}

\def\H{\mathcal{H}}
\def\e{\mathsf{e}}
\def\n{\mathsf{n}}
\def\d{\mathsf{d}_\square}
\def\Q{\mathcal{Q}}
\def\p{\mathsf{p_3}}
\def\w{\mathsf{w_3}}
\def\T{\mathcal{T}}
\def\B{\mathfrak{B}}
\newcommand\NN{\mathbb N}
\newcommand\RR{\mathbb R}

\newcommand\inner[2]{\langle #1, #2 \rangle}

\title[Efficient PTAS for the genus of dense graphs]{Efficient polynomial-time approximation scheme \\ for the genus of dense graphs}

\author{Yifan Jing}
\address{%
Mathematical Institute\\
University of Oxford\\
Oxford, UK}
\email{jing@maths.ox.ac.uk}

\author{Bojan Mohar}
\address{%
Department of Mathematics\\
Simon Fraser University\\
Burnaby, BC, Canada}
\email{mohar@sfu.ca}

\thanks{An extended abstract of the preliminary version of this paper appeared in ``59th IEEE Annual Symposium on Foundations of Computer Science, FOCS 2018, Paris, France, October 7--9, 2018", pp. 719--730.}
\thanks{Supported in part by the NSERC Discovery Grant R832714 (Canada),
and by the Research Project N1-0218 of ARIS (Slovenia).}
\thanks{On leave from IMFM \& FMF, Department of Mathematics, University of Ljubljana.}%

\date{}

\begin{document}

\begin{abstract}
The main results of this paper provide an Efficient Polynomial-Time Approximation Scheme (EPTAS) for approximating the genus (and non-orientable genus) of dense graphs. By \emph{dense} we mean that $|E(G)|\ge \alpha \, |V(G)|^2$ for some fixed $\alpha>0$. While a constant-factor approximation is trivial for this class of graphs, approximations with factor arbitrarily close to 1 need a sophisticated algorithm and complicated mathematical justification. More precisely, we provide an algorithm that for a given (dense) graph $G$ of order $n$ and given $\varepsilon>0$, returns an integer $g$ such that $G$ has an embedding in a surface of genus $g$, and this is $\varepsilon$-close to a minimum genus embedding in the sense that the minimum genus $\g(G)$ of $G$ satisfies:
$\g(G)\le g\le (1+\varepsilon)\g(G)$. The running time of the algorithm is
$O(f(\varepsilon)\,n^2)$, where $f(\cdot)$ is an explicit function.
Next, we extend this algorithm to also output an embedding (rotation system) whose genus is $g$. This second algorithm is an Efficient Polynomial-time Randomized Approximation Scheme (EPRAS) and runs in time $O(f_1(\varepsilon)\,n^2)$.

Our algorithms are based on the analysis of minimum genus embeddings of quasirandom graphs.
We use a general notion of quasirandom graphs \cite{LoSos}.
We start with a regular partition obtained via an algorithmic version of the Szemer\'edi Regularity Lemma (due to Frieze and Kannan \cite{FK} and to Fox, Lov\'asz, and Zhao \cite{FLZ,FLZ2}).
We then partition the input graph into a bounded number of quasirandom subgraphs, which are preselected in such a way that they admit embeddings using as many triangles and quadrangles as faces as possible. Here we provide an $\varepsilon$-approximation $\nu(G)$ for the maximum number of edge-disjoint triangles in $G$. The value $\nu(G)$ can be computed by solving a linear program whose size is bounded by certain value $f_2(\varepsilon)$ depending only on $\varepsilon$. After solving the linear program, the genus can be approximated (see Corollary \ref{cor:genus partition}). The proof of this result is long and will be of independent interest in topological graph theory.
\end{abstract}

\maketitle

%\tableofcontents

%\newpage

\section{Approximating the genus --- Overview}%

\subsection{The graph genus problem}

For a simple graph $G$, let $\g(G)$ be the {\em genus} of $G$, that is, the minimum $h$ such that $G$ embeds into the orientable surface $\mathbb{S}_h$ of genus $h$, and let $\ng(G)$ be the {\em non-orientable genus} of $G$ which is the minimum $c$ such that $G$ embeds into the non-orientable surface $\mathbb{N}_c$ with crosscap number $c$. If orientability of a surface is not a concern, the notion of the \emph{Euler genus}, which is defined as $\widehat \g(G) = \min\{2\g(G),\tilde \g(G)\}$, can be used. The genus is a natural measure how far is $G$ from being planar. Determining the genus of a graph is one of fundamental problems in graph theory with wide range of applications in computing and design of algorithms. Algorithmic interest comes from the fact that graphs of bounded genus share important properties with planar graphs and thus admit efficient algorithms for many problems that are difficult for general graphs \cite{Verdiere17, AlgsPlanar_Dagstuhl16, Parametrized_Book15}.

The genus of graphs played an important role in the developments of graph theory through its relationship to the \emph{Heawood problem} asking what is the largest chromatic number of graphs embedded in a surface of genus $g$. This problem was eventually reduced to the genus computation for complete graphs and resolved by Ringel and Youngs in 1968, see \cite{Ri74}. Further importance of the genus became evident in the seminal work on graph minors by Robertson and Seymour with developments of structural graph theory that plays an important role in stratification of complexity classes \cite{Parametrized_Book15}.

The \emph{genus problem} is the computational task of deciding whether the genus of a given graph $G$ is smaller than a given integer $k$. The question about its computational complexity was listed among the 12 outstanding open problems in the monograph by Garey and Johnson \cite{GaJo79} in 1979. Half of these problems were resolved by 1981, three of them (including graph isomorphism) are still unresolved, while three of them have been answered with considerable delay. The genus problem was among the latter three problems.  It was resolved in 1989 when Thomassen \cite{Th89} proved that it is NP-complete. Later, Thomassen simplified his proof in \cite{Th93} by showing that the question whether $G$ triangulates an (orientable) surface is NP-complete. In 1997 he also proved that the genus problem for cubic graphs is NP-complete \cite{Th97}. Mohar \cite{Mo01} proved that the genus is NP-complete even if we restrict our attention to \emph{apex graphs}, i.e. graphs which become planar by removing a single vertex.

Measuring graphs by their genus is \emph{fixed parameter tractable}. It follows from the Robertson and Seymour theory of graph minors (and their $O(n^3)$ algorithm for testing $H$-minor inclusion for any fixed graph $H$ \cite{GM13, GM22}) that for every fixed $k$, there is an $O(n^3)$ algorithm for testing whether a given graph $G$ has genus at most $k$. The time complexity in their cubic-time algorithm involves a huge constant depending on $H$, and the algorithm needs the (unknown) list of the forbidden minors for genus $k$. Notably, this is ``an impossible task" since the number of surface obstructions is huge (see, e.g., \cite{Myrvold18} for the up to date results about the surface of genus 1). Moreover, the Robertson--Seymour theory has a non-constructive element. The constants involved in their estimates about forbidden minors are not computable through their results. This deficiency was repaired with the results of Mohar \cite{Mo96,Mo99}, who found a linear-time algorithm for embedding graphs in any fixed surface.  His result generalizes the seminal linear-time algorithms for planarity testing by Hopcroft and Tarjan \cite{HT74} and by Booth and Lueker \cite{BL76}. It also generalizes to any surface the linear time algorithms that actually construct an embedding in the plane \cite{CNAO85} or find a Kuratowski subgraph when an embedding does not exist \cite{Wi84}. Mohar's algorithm gives a constructive proof for the finite number of forbidden minors for surface embeddability. The price paid for this is that the algorithms are complicated and hard to implement. A different linear-time FPT algorithm based on structural graph theory (reducing a graph to have bounded tree-width) has been found by Kawarabayashi, Mohar, and Reed \cite{KaMoRe08}. This algorithm includes as a subroutine a linear-time algorithm for computing the genus of graph of bounded tree-width, which turned out to be a difficult task by itself.

A large body of research has been done on approximating the genus by means of polynomial-time algorithms. Graphs whose genus is $\Theta(n)$ ($n$ being the number of vertices) admit a constant factor approximation algorithm. This is an easy consequence of Euler's formula, see \cite{ChKaKa97}. This case includes graphs whose (average) degree is at least $d$ for some $d>6$.

For graphs of bounded degree, $\Delta = \Delta(G)\le \Delta_0$, other approaches have been found. Chen et al.\ \cite{ChKaKa97} describe an algorithm with approximation factor $O(\sqrt{n})$.  Chekuri and Sidiropoulos \cite{ChSi13} found a polynomial-time algorithm that returns an embedding into a surface whose Euler genus is at most $(\Delta\, \widehat{\g}(G) \log n)^{O(1)}$. Here the approximation factor depends on $\Delta$, polylog factor in $n$ and polynomial factor of the Euler genus itself.

Some other results give additional insight into approximating the genus when the average degree is bounded. For example, the aforementioned paper of Mohar \cite{Mo01} yields a polynomial-time constant factor approximation for the genus of apex graphs (whose maximum degree can be arbitrarily large, but their average degree is less than 8). This result was extended to $k$-apex graphs in \cite{KaSi15}.

Kawarabayashi and Sidiropoulos \cite{KaSi15} removed the dependence on the maximum degree needed in Chekuri and Sidiropoulos approximation. With a very clever approach they were able to design a polynomial-time algorithm that approximates the Euler genus of any graph within a factor of $O(\widehat{\g}^{255}\log^{189}n)$. A corollary of their result is that the genus can be approximated within factor $O(n^{1-\alpha})$ for some constant $\alpha>0$, see \cite{KaSi15}.  A predecessor to this result was published by Makarychev, Nayyeri, and Sidiropoulos \cite{MaNaSi13}, who proved that for a graph $G$ possessing a Hamiltonian path (which, unfortunately, needs to be given as part of the input), one can efficiently approximate the Euler genus within factor $(\g(G) \log n)^{O(1)}$. Here the quality of approximation depends on the orientable genus together with a $polylog(n)$ factor.

A while ago, the second named author outlined an approach to approximating the genus of graphs that splits the problem in four different regimes, where he suspects different kind of results can be obtained and different kind of methods will be used.

\begin{itemize}
\item[(1)]
{\bf Spherical density.} In this regime, we consider graphs with (average) degree at most 6. While this regime is probably most interesting, there is no evidence that there are efficient constant-factor approximation algorithms, so resolving this is the major open problem. Moreover, it has been conjectured by the second named author that approximation in this range is hard.
\begin{conjecture}[Mohar, 2018]
Approximating the genus of cubic graphs is APX-hard.
\end{conjecture}
\item[(2)]
{\bf Bounded (average) degree with hyperbolic density.} In this regime, we consider graphs whose degree (or average degree) is bounded above by a constant $\Delta$ and is bounded below by a smaller constant $\delta>6$. It is easy to see that the genus in this range can be approximated within a constant. However, any improvements of that ``trivial" constant may be hard to obtain. Moreover, the approximation may still be APX-hard.
\item[(3)]
{\bf Intermediate density range.} In this regime, we consider graphs with unbounded average degree and $o(n^2)$ edges. While we expect that possible hardness of approximation from range (2) and PTAS approximability from range (4) could be extended to some extent into this range, it is not clear at all how much this would work.
\item[(4)]
{\bf Dense graphs.} In this regime, we consider \emph{dense graphs}. These are graphs with quadratic number of edges, i.e. $|E(G)|\ge \alpha |V(G)|^2$, where $\alpha>0$ is a constant. This paper provides a complete solution for this case.
\end{itemize}

\subsection{Our results}

The main results of this paper provide an Efficient Polynomial-Time Approximation Scheme (EPTAS) for approximating the genus of dense graphs. A graph is \emph{$\alpha$-dense} if $\e(G):=|E(G)|\ge \alpha n^2$. By saying a graph $G$ is \emph{dense} we mean it to be $\alpha$-dense for some fixed $\alpha>0$. While a constant-factor approximation is trivial for this class of graphs, approximations with factor arbitrarily close to 1 provided in this paper need a sophisticated algorithm and complicated mathematical justification.

Given a (dense) graph $G$ of order $n$ and the allowed approximation error $\varepsilon>0$, we want to find an integer $g$ and an embedding of $G$ into a surface of genus $g$ which is close to a minimum genus embedding, i.e. $\g(G)\le g\le (1+\varepsilon)\g(G)$.  The problem is equivalent to the following one, where the assumption on density is left out.

\begin{quote}
	{\sc Approximating Genus Dense}.\\
	\emph{Input:} A graph $G$ of order $n$ and a real number $\varepsilon>0$.\\
	\emph{Output:} An integer $g$ and either a conclusion that $\g(G)\le g<\varepsilon n^2$, or that $\g(G)\le g\le (1+\varepsilon)\g(G)$.
\end{quote}

To see the equivalence, by Euler's formula, we have
\[
\g(G)=\frac{1}{2}(|E(G)|-|F(G)|-|V(G)|+2),
\]
where $|F(G)|$ is the number of faces in the minimum genus embedding. As each face has at least three edges, one have $3|F(G)| \leq 2|E(G)|$. Suppose $\alpha$ is the edge density of $G$, then $\g(G)\geq (\frac16+o(1))\alpha n^2$. By choosing $\varepsilon<\alpha/6$, we will see that the algorithm always output an integer $g$ with $\g(G)\le g<(1+\varepsilon)\g(G)$.

% It is easy to see that (after appropriate rescaling of $\varepsilon$) this problem is equivalent to the following one, where the assumption on density is left out.

% \begin{quote}
% 	{\sc Approximating Genus Dense}.\\
% 	\emph{Input:} A graph $G$ of order $n$ and a real number $\varepsilon>0$.\\
% 	\emph{Output:} An integer $g$ and either a conclusion that $\g(G)\le g<\varepsilon n^2$, or that $\g(G)\le g\le (1+\varepsilon)\g(G)$.
% \end{quote}

In order to obtain EPTAS for the genus of dense graphs, we outline an algorithm whose time complexity is $O(f(\varepsilon) n^2)$.
%$O(f(\varepsilon) n^{O(1)})$, where $f:\RR^+\to\RR^+$ is an explicit positive function. 
%In fact, the polynomial dependence on $n$ in our algorithm is quadratic.

\begin{theorem}
\label{thm:main algorithm genus}
The problem {\sc Approximating Genus Dense} can be solved in time $O(f(\varepsilon)\,n^2)$, where $f:\RR^+\to\RR^+$ is an explicit positive function.
\end{theorem}

The dependence on $\varepsilon^{-1}$ in the time complexity of our algorithms is super-exponen\-tial. In fact, the value $f(\varepsilon)$ is a tower of exponents of height $O(\varepsilon^{-1})$. There are two steps where non-polynomial dependence on $\varepsilon^{-1}$ occurs. The obvious one is when we apply the Regularity Lemma. The second one appears when treating ``small'' graphs. Namely, $n$ is considered small if it is smaller than $n_0$, where $n_0=n_0(\varepsilon)$ is the bound from Theorem \ref{thm:genus} and it grows very fast in terms of $\varepsilon^{-1}$. This case has linear complexity in terms of $n$ \cite{Mo96,Mo99}, but it involves a large constant factor which is increasing super-exponentially fast with $\varepsilon^{-1}$. This is where we are prevented of designing an FPTAS.

Our second computational result extends the previous one by constructing an embedding whose genus is close to the minimum genus. Formally, we consider the problem:

\begin{quote}
	{\sc Approximate Genus Embedding Dense}.\\
	\emph{Input:} A graph $G$ of order $n$ and a real number $\varepsilon>0$.\\
	\emph{Output:} Rotation system of a 2-cell embedding of $G$, whose genus $g$ is close to $\g(G)$: either $\g(G) \le g \le (1+\varepsilon)\g(G)$, or $\g(G)\le g <\varepsilon n^2$.
\end{quote}

As a solution we provide an Efficient Polynimial-time Randomized Approximation Scheme (EPRAS) \cite{Va03}.

\begin{theorem}
\label{thm:main algorithm embedding}
There is a randomized algorithm for {\sc Approximate Genus Embedding Dense} that returns an embedding of the input graph $G$ of genus $g$, such that either $\g(G) \le g \le (1+\varepsilon)\g(G)$, or $\g(G)\le g <\varepsilon n^2$.  The time spent by the algorithm is $O(f_1(\varepsilon)\,n^2)$, where $f_1:\RR^+\to\RR^+$ is an explicit positive function.
\end{theorem}

There are two parts in the embedding algorithm that are nondeterministic. One of them uses random partition of the edges of $G$. This part can be derandomized (yielding a cubic polynomial dependence on $n$), but for the other one we do not see how to derandomize it. This part finds a large matching in a 3-uniform (or 4-uniform) hypergraph, and existence of a large matching and its construction relies on the Lov\'asz Local Lemma (for which one could use a randomized algorithm by Moser and Tardos \cite{MoTa10}). We make use of another randomized solution involving R\"odl nibble \cite{RoTh96} which yields quadratic dependence on $n$.

Algorithms in Theorems \ref{thm:main algorithm genus} and \ref{thm:main algorithm embedding} are based on analysis of minimum genus embeddings of quasirandom graphs. We partition the input graph into a bounded number of quasirandom subgraphs, which are preselected in such a way that they admit embeddings using as many triangles and quadrangles as faces as possible. The starting partition is obtained through an algorithmic version of the Szemer\'edi Regularity Lemma (due to Frieze and Kannan \cite{FK} and to Fox, Lov\'asz, and Zhao \cite{FLZ,FLZ2}).

We use the notion of quasirandomness inspired by the seminal paper of Chung, Graham, and Wilson \cite{quasi} (see also \cite{C}), and we need it in two special cases of bipartite and tripartite graphs, respectively. In order to define it, we need some notation.

Let $G$ be a graph, and $X,Y\subseteq V(G)$. We define the \emph{edge density} between $X$ and $Y$ as the number 
$$d(X,Y) = \frac{\e(X,Y)}{|X|\cdot|Y|},$$ 
where $\e(X,Y)$ is the number of edges with one end in $X$ and another end in $Y$.
If $G$ is a (large) bipartite graph with balanced bipartition $V(G) = V_1\cup V_2$ (by \emph{balanced} we mean that $| |V_1|-|V_2| | \le 1$), we say that $G$ is \emph{$\varepsilon$-quasirandom} if for every $X\subseteq V_1$ and every $Y\subseteq V_2$, we have
$$
    |\e(X,Y) - |X| |Y| d(V_1,V_2)| \le \varepsilon |V_1| |V_2|.
$$
This is equivalent (see \cite{quasi}) to saying that the number of 4-cycles with vertices in $X\cup Y$ is close to what one would expect in a random bipartite graph with the same edge density, i.e. $d^4(V_1,V_2)|V_1|^2 |V_2|^2$, with an error of at most $4\varepsilon |V_1|^2 |V_2|^2$.

First, we prove the following result.

\begin{theorem}
\label{thm:genus QR bipartite}
For every $\varepsilon>0$ there is a constant $n_0$ such that the following holds.
Suppose that $G$ is a bipartite graph with balanced bipartition $\{V_1,V_2\}$. Suppose that $G$ is $\varepsilon$-quasirandom with edge density $d=d(V_1,V_2)>0$. If $0 < \varepsilon < (d/4)^{16}$ and $n = |V(G)| \ge n_0$, then
$$
   \tfrac{1}{4}(1-10\varepsilon) \e(G) \le \g(G) \le \tfrac{1}{4}(1+10\varepsilon) \e(G).
$$
\end{theorem}

The above theorem says, roughly speaking, that $G$ admits an embedding in some orientable surface in which almost all faces are quadrilaterals. More precisely, almost every edge is contained in two quadrangular faces.

A similar result holds for \emph{tripartite $\varepsilon$-quasirandom graphs}. Here we have three almost equal parts $V_1,V_2,V_3$, and the graph between any two of them is bipartite $\varepsilon$-quasirandom. Here we only need the corresponding embedding result when the densities between the three parts are the same.

\begin{theorem}
\label{thm:genus QR tripartite}
For every $\varepsilon>0$ there is a constant $n_0$ such that the following holds.
Suppose that $G$ is a tripartite $\varepsilon$-quasirandom graph with edge densities $d=d(V_1,V_2)=d(V_1,V_3)=d(V_2,V_3)>0$. If $0 < \varepsilon < (d/4)^{12}$ and $n = |V(G)|\ge n_0$, then
$$
   \tfrac{1}{6}(1-10\varepsilon) \e(G) \le \g(G) \le \tfrac{1}{6}(1+10\varepsilon) \e(G).
$$
\end{theorem}

Similarly as for Theorem \ref{thm:genus QR bipartite}, the outcome of the above theorem is that $G$ admits an embedding in which almost every edge is contained in two triangular faces.

Theorems \ref{thm:genus QR bipartite} and \ref{thm:genus QR tripartite} have extensions to multipartite case with possibly non-equal edge densities and non-empty quasirandom graphs in the parts of the vertex partition.  For this extension see the main part of the paper (Theorem \ref{thm:quasimultigenus}).

Proofs of Theorems \ref{thm:genus QR bipartite} and \ref{thm:genus QR tripartite} build on the approach introduced by Archdeacon and Grable \cite{dense} and R\"odl and Thomas \cite{genus}.  The main ingredient is to find two disjoint almost perfect matchings in a 3-uniform (or in a 4-uniform) hypergraph associated with short cycles in $G$. One particular difficulty is that there may be too many short cycles, in which case the matchings obtained from these hypergraphs may not form a set which could be realized as facial cycles of an embedding of the graph. This has to be dealt with accordingly. The proof uses an old result of Pippenger and Spencer \cite{PS} about existence of large matchings in hypergraphs.

From the above expressions about the genus of quasirandom graphs, there is just one major step left. We show that $G$ can be partitioned into a constant number of bipartite and tripartite $\varepsilon$-quasirandom subgraphs such that almost all triangles of $G$ belong to the tripartite $\varepsilon$-quasirandom subgraphs in the partition. For each of these $\varepsilon$-quasirandom subgraphs, Theorem \ref{thm:genus QR tripartite} or Theorem \ref{thm:genus QR bipartite} can be applied. The main result is the following version of the Szemer\'edi Regularity Lemma. See Sections \ref{sec:cut metric}--\ref{sec:notation} for necessary definitions. 

\begin{theorem}
\label{thm:partition}
There exists a computable function $s: \NN\times \RR^+ \to \NN$ such that the following holds.
For every $\varepsilon>0$ and every positive integer $M$ there is an integer $K$, where $M\le K \le s(M,\varepsilon)$ such that every graph of order $n\ge M$ has an equitable partition of its vertices into $K$ parts, $V(G)=V_1\cup\cdots\cup V_K$, and $G$ admits a partition into $O(K^2)$ bipartite $\varepsilon$-quasirandom subgraphs $G_{ij}$ $(1\le i<j\le K)$ with $V(G_{ij})=V_i\cup V_j$, and into $O(K^3)$ tripartite $\varepsilon$-quasirandom subgraphs $G_{ijk}$ $(1\le i<j<k\le K)$ with $V(G_{ijk})=V_i\cup V_j\cup V_k$ with equal densities between their parts, and one additional subgraph $G_0$ with at most $\varepsilon n^2$ edges. Moreover, the union of all bipartite constituents $G_{ij}$ is triangle-free.
\end{theorem}

To obtain such a partition $V(G)=V_1\cup\cdots\cup V_K$ we start with an $\varepsilon$-regular partition obtained from the Szemer\'edi Regularity Lemma. Such a partition can be constructed in quadratic time by using an algorithm of Fox, Lov\'asz, and Zhao \cite{FLZ2}. The edges in ``irregular pairs'' and all edges in subgraphs $G[V_i]$ ($1\le i\le K$) are put into the subgraph $G_0$. All remaining edges belong to bipartite subgraphs joining pairs $V_i$ and $V_j$ ($1\le i<j\le K$). Let $d_{ij}$ be the edge density for each such bipartite subgraph. We represent the partition by a weighted graph $H$ on vertices $\{1,\dots,K\}$, where each edge $ij$ has weight $d_{ij}$.

Let $\T$ be the set of all triangles (with positive edge weights) in the quotient graph $H$. For every triangle $T=abc\in\T$, let $d(T)=\min\{d_{ab},d_{bc},d_{ac}\}$. Now we consider the following linear program with indeterminates $\{t(T)\mid T\in\T\}$:

\begin{equation}\label{eq:lp}
\begin{split}
&\nu^*(H)=\max\sum_{T\in\T}t(T),\\
&\sum_{T\ni ij,T\in\T}t(T)\leq d_{ij}, \quad \text{ for every edge }ij\text{ of }H,\\
&t(T)\geq0,\quad \text{ for every }T\in\T.
\end{split}
\end{equation}

We consider an optimum solution $(t(T)\mid T\in\T)$ of this linear program.
For each $T=abc\in \T$ we now define $G_T$ as a subgraph of $G[V_a\cup V_b\cup V_c]$ by taking a random set of edges with density $t(T)$ from each of the three bipartite graphs between $V_a,V_b,V_c$. The edges between the sets $V_i$ and $V_j$ ($1\le i<j\le K$) that remain after removing all tripartite subgraphs $G_T$ form bipartite quasirandom subgraphs.

From the Partition Theorem \ref{thm:partition} it is not hard to see that
$$
  \g(G) \le \sum_{i,j} \g(G_{ij}) + \sum_{i,j,k} \g(G_{ijk}) + nK^2 + \varepsilon n^2.
$$
By using Theorems \ref{thm:genus QR bipartite} and \ref{thm:genus QR tripartite} we derive the main result:

\begin{corollary}
\label{cor:genus partition}
For every $\varepsilon>0$ there is a constant $n_0$ such that the following holds.
Let $G$ be a graph that is partitioned as stated in Theorem \ref{thm:partition}, let $\nu=\nu^*(H)$ be the optimum value of the linear program (\ref{eq:lp}), and let $s:\NN\times \RR^+ \to \NN$ be the function from Theorem \ref{thm:partition}. If $n = |V(G)|\ge n_0$, then the genus of $G$ satisfies:
$$
  \frac{1}{4} (1-\varepsilon)\left(\e(G) - \frac{\nu n^2}{K^2}\right) \le \g(G) \le
   \frac{1}{4} (1+\varepsilon) \left(\e(G) - \frac{\nu n^2}{K^2}\right) + nK^2 + \varepsilon n^2.
$$
\end{corollary}

The value of $n_0$ in Corollary \ref{cor:genus partition} comes from several sources. One is the use of Regularity Lemma, another one comes from using hypergraph matching results (see Theorems \ref{thm:hypergraphmatching} and \ref{thm: packing weighted graph}), and yet another one occurs from the fine-tuning of the parameters in Section \ref{subsect:multipartite QR}.

An application of Theorems \ref{thm:genus QR bipartite} and \ref{thm:genus QR tripartite} requires that the densities are not too small. If this is not the case, we simply add the edges between pairs $V_i,V_j$, whose density is too small, to $G_0$.

In order to apply the corollary to obtain an $\varepsilon$-approximation to the genus, we use the corollary with the value $\tfrac{1}{2}\varepsilon$ playing the role of $\varepsilon$.  If $n\ge \Theta(s^2(4\varepsilon^{-1}, \varepsilon)\cdot \varepsilon^{-1})$, then the last two terms in the corollary are bounded by $\tfrac{3}{4}\varepsilon n^2$. Now, if $\varepsilon$ is much smaller than the lower bound $\alpha$ on the density of $G$, we get an $\varepsilon$-approximation of the genus of $G$.

Although we have not mentioned anything about the nonorientable genus before, the same results hold for the nonorientable genus, where all formulas about the genus need to be multiplied by 2.

\subsection{Overview of the algorithms}

On a high level, our EPTAS for {\sc Approximating Genus Dense} works as follows.

\medskip

{\sc Phase 0.} Check whether the graph is dense enough: If $\e(G) \le \varepsilon n^2$, we return the information that $\g(G)<\varepsilon n^2$ and stop.

{\sc Phase 1.} Let $M=2\varepsilon^{-1}$ and let $M'=s(m,\varepsilon)$ where the function $s$ is from Theorem \ref{thm:partition}.  If $|G|\leq M'^2 \varepsilon^{-1}$, we compute the genus of $G$ exactly and return the result. Otherwise, we proceed with the next step.

{\sc Phase 2.} We find a regular partition of $G$ into $K$ parts, where $M\le K\le M'$, and according to Theorem \ref{thm:partition} partition $G$ into edge-disjoint subgraphs $G_0,G_1,\dots,G_N$, where $N=O(K^3)$, each of them except $G_0$ of order $\Theta(n/K)$ such that the following holds:
\begin{itemize}
\item[(i)] $G_0$ has at most $\tfrac{1}{2}\varepsilon n^2$ edges.
\item[(ii)] All other subgraphs are either bipartite or tripartite $\varepsilon$-quasirandom.
\item[(iii)] The union of bipartite subgraphs contains no triangles.
\end{itemize}

{\sc Phase 3.} Determine the densities $d_{ij}$, $1\le i<j\le K$, and solve the linear program (\ref{eq:lp}).
Let $\nu=\nu^*(H)$ be the optimal value computed.
Return the value $g = \frac{1}{4}\e(G) - \frac{\nu n^2}{4K^2}$.

\medskip

The heart of the algorithm lies in {\sc Phase 2}. However, {\sc Phase 3} is the most challenging mathematical part and has a complicated justification. For the partition of $G$ into $G_0,G_1,\dots, G_N$ we could use the algorithmic version of the Szemer\'edi Regularity Lemma due to Frieze and Kannan \cite{FK}. But it is more convenient to use a recent strengthening of Frieze-Kannan partitions due to Fox, Lov\'asz, and Zhao \cite{FLZ}. Their result provides an \emph{$\varepsilon$-regular partition} (in the sense of Szemer\'edi) of $V(G)$ into sets $V_1,\dots, V_K$ of size $n/K$ (we neglect rounding of non-integral values as they are not important for the exposition) such that the majority of pairs $(V_i,V_j)$ ($1\le i<j\le K$) are \emph{$\varepsilon$-regular}. Each such pair induces a bipartite $\varepsilon$-quasirandom graph. The edges in pairs that are not $\varepsilon$-regular can be added to $G_0$ together with all edges in $\bigcup_{i=1}^K G[V_i]$. So from now on, we assume that all edges of the graph are in $\varepsilon$-regular pairs $(V_i,V_j)$ ($1\le i<j\le K$).

The second, most difficult step, is to analyse the quotient graph determined by the partition. In this step we use a linear programming approach to find for each triple $T=(i,j,k)$, $1\le i<j<k\le K$, the number $t(T)\ge0$ and an $\varepsilon$-quasirandom subgraph $G_T\subseteq G[V_i\cup V_j\cup V_k]$ with $t(T)(n/K)^2$ edges between each pair $(V_i,V_j), (V_i,V_k)$, and $(V_j,V_k)$. The graphs $G_T$ are then used to obtain as many triangular faces as possible (up to the allowed error) for the embedding of $G$. The edges that remain form quasirandom bipartite parts $G_{ij}$ between pairs $(V_i,V_j)$ ($1\le i<j\le K$). We use those to obtain as many quadrangular faces as possible (up to the allowed error) for the embedding of $G$.

Finally, the near-triangular embeddings of all $G_T$ and near-quadrangular embeddings of all $G_{ij}$ are used to produce a near-optimal embedding of $G$. The description of this part is in the main part of the paper.

Let us now comment on the main issues in the algorithmic part and in the theoretical justification. First, we use known regularity partition results to find a partition $V_1,\dots,V_K$. The algorithm runs in quadratic time with a decent (but superexponential) dependence on $1/\varepsilon$. The linear programming part to determine triangle densities $t(T)$ is done on a constant size linear program and a rounding error of magnitude $O(\varepsilon)$ is allowed. Having gathered all the information about the required edge densities in the partition, the partition of $G$ into subgraphs $G_T$ and $G_{ij}$ uses a randomized scheme, although derandomization is possible. For the computation of the approximate value for the genus (Corollary \ref{cor:genus partition}), the partition is not needed. We just need to know that it exists, and we need to know the edge densities between the regular pairs of the partition. Thus, this part is deterministic.

For the justification that the graphs $G_T$ and $G_{ij}$ admit almost triangular and almost quadrangular embeddings, respectively, we use the quasirandomness condition. The proof is based on a theorem by Pippenger and Spencer \cite{PS} giving a large matching in a dense 3- or 4-uniform hypergraph (respectively). The hyperedges in the hypergraph correspond to cycles of length 3 and 4 (respectively) in the considered subgraph $G_T$ or $G_{ij}$. Two such matchings are needed in order to combine them into an embedding of the graph, most of whose faces will be the triangles or quadrangles of the two hypergraph matchings. Quasirandomness is used to show that such matchings exists and that they have additional properties needed for them to give rise to an embedding. To obtain such a matching, we can follow the proof of Pippenger and Spencer, but the proof uses the Lov\'asz Local Lemma. In order to construct such a matching, we may apply the algorithmic version of the Lov\'asz Local Lemma that was obtained by Moser and Tardos \cite{MoTa10}. Alternatively, we may apply the randomized algorithm of R\"odl and Thoma \cite{RoTh96} which uses the R\"odl nibble method. A similar algorithm was obtained by Spencer \cite{S95}. Both of these latter algorithms use greedy selection and run in quadratic time, but they are both randomized. This is the essential part where we are not able to provide corresponding derandomized version.

%%%%%%%%%%%%%%%%%%%%%%%%%%%%%%%%%%%%%%%%%%%%%%%%%%%%
\section{Preliminaries}

This section contains basic definitions and theoretical background needed for the rest of the paper.
\subsection{Minimum genus embeddings}
Given a simple graph $G$, let $\g(G)$ be the {\em genus} of $G$, that is, the minimum $h$ such that $G$ embeds into the orientable surface $\mathbb{S}_h$ of genus $h$, and let $\ng(G)$ be the {\em non-orientable genus} of $G$ which is the minimum $c$ such that $G$ embeds into the non-orientable surface $\mathbb{N}_c$ with crosscap number $c$. By a surface we mean is a compact two-dimensional manifold without boundary. We say $G$ is {\em 2-cell embedded} in a surface $S$ if each face of $G$ is homeomorphic to an open disk. We say an embedding of $G$ is {\em triangular} if every face is bounded by a triangle, and an embedding is {\em quadrangular} if every face is bounded by a cycle of length $4$.

Given a graph $G$, determining the genus of $G$ is one of the fundamental problems in topological graph theory. Youngs \cite{Y} showed that the problem of determining the genus of a connected graph $G$ is the same as determining a $2$-cell embedding of $G$ with minimum genus. The same holds for the non-orientable genus \cite{PPPV}. For further background on topological graph theory, we refer to \cite{top}.

Now we focus on the $2$-cell embeddings of a graph $G$. We say $\Pi=\{\pi_v\,|\,v\in V(G)\}$ is a {\em rotation system} if for each vertex $v$, $\pi_v$ is a cyclic permutation of the edges incident with $v$. The {\em Heffter-Edmonds-Ringel rotation principle} \cite[Theorem 3.2.4]{top} shows that every $2$-cell embedding of a graph $G$ in an orientable surface is determined (up to homeomorphisms of the surface) by its rotation system. The rotation system $\Pi$ determines the set of \emph{$\Pi$-facial walks}, each of which corresponds to the boundary of a face in the topological embedding. We will refer to these as the \emph{faces} of $\Pi$. For $2$-cell embeddings we have the famous Euler's Formula.

\begin{theorem}[Euler's Formula]
Let $G$ be a graph that is 2-cell embedded in a surface $S$. If $G$ has $n$ vertices, $e$ edges and $f$ faces in $S$, then
\begin{equation*}
\chi(S)=n-e+f.
\end{equation*}
\end{theorem}
Here $\chi(S)$ is the {\em Euler characteristic} of the surface $S$, where $\chi(S)=2-2h$ when $S=\mathbb{S}_h$ and $\chi(S)=2-c$ when $S=\mathbb{N}_c$.

Let $G$ be a simple graph. The {\em corresponding digraph} of $G$ is a random simple digraph $\mathcal{D}$ obtained from $G$ by randomly orienting each edge. Specifically, each digraph $D\in\mathcal{D}$ has $V(D)=V(G)$ and if $uv\in E(G)$ then either $\overrightarrow{uv}$ or $\overrightarrow{vu}$ is an edge of $D$, each with probability $1/2$, and the two events are exclusive. For a digraph $D$, we let $D^{-1}$ be the digraph obtained from $D$ by replacing each arc $\overrightarrow{xy}$ with the reverse arc $\overrightarrow{yx}$.

Given a digraph, a {\em blossom of length} $l$ with {\em center} $v$ and {\em tips} $\{v_1,v_2,\dots,v_l\}$ is a set $\mathcal{C}$ of $l$ directed cycles $\{ C_1,C_2,\dots,C_l\}$, where $\overrightarrow{v_{i}v},\overrightarrow{vv_{i+1}}\in C_i$, for $i=1,2,\dots,l$, with $v_{l+1}=v_1$. A \emph{$k$-blossom} is a blossom, all of whose elements are directed $k$-cycles. A blossom of length $l$ is {\em simple} if either $l\geq3$ or $l=2$ and $C_1\neq C_2^{-1}$. In this paper, we mainly consider $3$-blossoms and $4$-blossoms (see Figure \ref{fig:b}).

\begin{figure}[ht]
\centering
\begin{minipage}[t]{0.5\textwidth}
\centering
\includegraphics[width=2.8in]{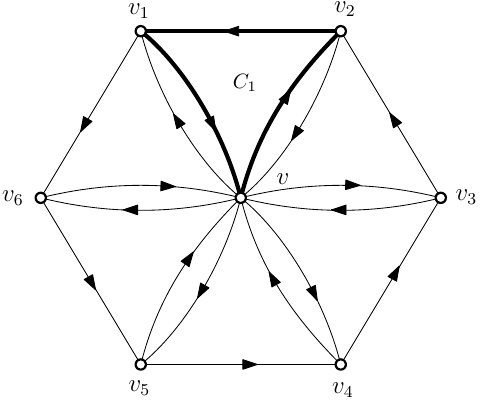}
\end{minipage}\hfill\begin{minipage}[t]{0.5\textwidth}
\centering
\includegraphics[width=2.8in]{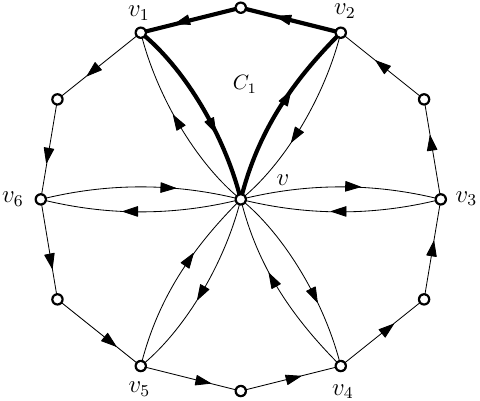}
\end{minipage}
\caption{A $3$-blossom and a $4$-blossom of length $6$. The cycle $C_1$ is shown by bold edges.}\label{fig:b}
\end{figure}

Let $\mathcal{C}$ be a family of arc-disjoint closed trails in $D\cup D^{-1}$. We say that $\mathcal{C}$ is {\em blossom-free} if no subset of $\mathcal{C}$ forms a blossom centered at some vertex. The following lemma is the main tool used to construct embeddings of a graph from blossom-free families of cycles.

\begin{lemma}\label{lem:rotation}
Let $G$ be a graph and let $D$ be the corresponding digraph. Suppose that $\mathcal{C}_1$ and $\mathcal{C}_2$ are sets of arc-disjoint cycles in $D$ and $D^{-1}$ (respectively) such that their union $\mathcal{C}_1\cup\mathcal{C}_2$ is blossom-free in $D\cup D^{-1}$. Then there exists a rotation system $\Pi$ of $G$ such that every cycle in $\mathcal{C}_1\cup\mathcal{C}_2$ is a face of $\Pi$.
\end{lemma}

\begin{proof}
We say a rotation system $\Pi=\{\pi_v\,|\,v\in V(G)\}$ is \emph{compatible} with $\mathcal{C}_1$ if for each cycle $C\in \mathcal{C}_1$ and each vertex $v\in V(C)$ with incident edges $e,f\in E(C)$ whose orientation in $D$ is such that $e$ precedes $f$ on $C$, we have $\pi_v(f)=e$. Under the assumptions of the lemma, it is easy to see that there exists a rotation system that is compatible with $\mathcal{C}_1$ and is also compatible with $\mathcal{C}_2$. Such rotation system will have every cycle in $\mathcal{C}_1 \cup \mathcal{C}_2$ as a face.
\end{proof}

Given a 2-cell embedding with rotation system $\Pi$, let $f_k(\Pi)$ be the number of faces of length $k$ in $\Pi$. For every $\varepsilon>0$, an {\em $\varepsilon$-near triangular embedding} and an {\em $\varepsilon$-near quadrangular embedding} is a rotation system of $G$ such that $3f_3(\Pi)\geq2(1-\varepsilon)|E(G)|$ and $4f_4(\Pi)\geq2(1-\varepsilon)|E(G)|$, respectively. Lemma \ref{lem:rotation} shows that in order to find an $\varepsilon$-near triangular embedding (or an $\varepsilon$-near quadrangular embedding), it suffices to find a large set of directed triangles (or directed 4-cycles) in $D\cup D^{-1}$ which is also blossom-free.

\subsection{Matchings in hypergraphs}

In order to construct an $\varepsilon$-near triangular (or quadrangular) embedding, one way is that we view the edge-set of $G$ as a vertex-set of a hypergraph $\H$. We randomly orient the edges of $G$ and then consider  all directed triangles (or directed 4-cycles) in $G$ as hyperedges of $\H$. Note that $\H$ is a $3$-uniform (or $4$-uniform) hypergraph. The following result from \cite{PS} (see also \cite{FR,genus} where its current formulation appears) on hypergraph matching theory will play an important role for constructing near-optimal embeddings of graphs. Recall that a \emph{matching} in a hypergraph $\H$ is a set of pairwise disjoint edges of $\H$. A matching $M\subseteq E(\H)$ is \emph{$\varepsilon$-near perfect} if all but $\varepsilon|V(\H)|$ vertices of $\H$ are contained in the edges in $M$.

\begin{theorem}\label{thm:hypergraphmatching}
Let $\varepsilon>0$ be a real number and $d\geq2$ be an integer. Then there exist a positive real number $\delta$ and an integer $N_0$ such that for every $N\geq N_0$ the following holds. If $\Delta$ is a real number and if $\mathcal{H}$ is a $d$-uniform hypergraph with $|V(\mathcal{H})|=N$ such that
\begin{enumerate}
\item $|\{x\in V(\mathcal{H})\ |\ (1-\delta)\Delta\leq \deg(x)\leq(1+\delta)\Delta\}|\geq(1-\delta)N$,
\item for every $x,y\in V(\mathcal{H})$ with $x\neq y$, $|\{e\in E(\mathcal{H})\mid x,y\in e\}|<\delta\Delta$, and
\item at most $\delta N\Delta$ hyperedges of $\mathcal{H}$ contain a vertex $v\in V(\mathcal{H})$ with $\deg(v)>(1+\delta)\Delta$,
\end{enumerate}
then $\mathcal{H}$ has an $\varepsilon$-near perfect matching. Moreover, for every matching $M$ in $\mathcal{H}$, there exists an $\varepsilon$-near perfect matching $M^\prime$ in $\mathcal{H}$ with $M\cap M^\prime=\varnothing$.
%then $\mathcal{H}$ has a matching of size at least $(1-\varepsilon)N/d$. Moreover, for every matching $M$ in $\mathcal{H}$, there exists a matching $M^\prime$ in $\mathcal{H}$ with $M\cap M^\prime=\varnothing$, and with $|M^\prime|\geq(1-\varepsilon)N/d$.
\end{theorem}

For nearly regular uniform hypergraphs, R\"odl and Thoma provide a random greedy algorithm \cite{RoTh96} (see also \cite{S95} for a different proof) which can output an $\varepsilon$-near perfect matching in $\H$. It is of interest to mention that Ford et al.\ \cite{FGKMT} strengthened this result by removing the requirement that the hypergraph be nearly regular.
%\begin{theorem}\label{thm:algmatching}
%Let $d\geq2$ be an integer and $\varepsilon$ $(0<\varepsilon<1/10)$ be a real number. Then there exist $\delta>0$, and an integer $N_0$ such that the following holds. Suppose that $\Delta>0$, and $\H$ is a $d$-uniform hypergraph of order $N\geq N_0$ that satisfies
%\begin{enumerate}
%\item $|\deg(x)-\Delta|<\delta\Delta$, for every $x\in V(\H)$, and
%\item for every $x,y\in V(\mathcal{H})$, $|\{e\in E(H)\mid x,y\in e\}|<\delta\Delta$.
%\end{enumerate}
%Then there exists a random greedy algorithm which stops after yielding a matching of size $\lceil \tfrac{n}{d}(1-\varepsilon)\rceil$ in $\H$ almost surely, and the expected running time is $T_\varepsilon=\Theta((\tfrac{1}{\varepsilon})^{d-1})\tfrac{N}{d(d-1)}$.
%\end{theorem}

%In Section 4, we will use this algorithm to output an $\varepsilon$-near perfect matching of $\H$ if $\H$ satisfies conditions (1)--(3) in Theorem \ref{thm:hypergraphmatching}.

\subsection{Cut metric and regular partitions}
\label{sec:cut metric}

Szemer{\'e}di's  Regularity Lemma \cite{S} is one of the most powerful tools in understanding large dense graphs. Szemer{\' e}di first used the lemma in his celebrated theorem on long arithmetic progressions in dense subset of integers \cite{S2}. Nowadays, the regularity lemma has many connections to other areas of mathematics, for example, analysis, number theory and theoretical computer science. For an overview of applications, we refer to \cite{Ana, Tao}. Regularity Lemma gives us a structural characterization of graphs. Roughly speaking, it says that every large graph can be partitioned into a bounded number of parts such that the graphs between almost every pair of parts is random-like. To make this precise we need some definitions.

Let $G$ be a graph and $X,Y\subseteq V(G)$. We define the {\em edge density} $d(X,Y)=\e(X,Y)/(|X||Y|)$. We say the pair $(X,Y)$ is {\em $\varepsilon$-regular} if for all $X^\prime\subseteq X$ and $Y^\prime\subseteq Y$ with $|X^\prime|\geq\varepsilon|X|$ and $|Y^\prime|\geq\varepsilon|Y|$, we have $|d(X^\prime,Y^\prime)-d(X,Y)|<\varepsilon$. A partition $\mathcal{P}=\{V_i\}_{i=1}^K$ of $V(G)$ is \emph{equitable} if for every $1\leq i<j\leq K$ we have $\big||V_i|-|V_j|\big|\leq1$. An equitable vertex partition $\{V_1,\dots,V_{K}\}$ with $K$ parts is {\em $\varepsilon$-regular} (or \emph{regular} if $\varepsilon$ is clear from the context) if all but at most $\varepsilon K^2$ pairs of parts $(V_i,V_j)$ ($1\leq i<j\leq K$) are $\varepsilon$-regular.

\begin{theorem}[Szemer\'edi's Regularity Lemma]\label{thm:regularity}
There exists a computable function $s:\NN\times \RR^+\to\NN$ such that the following holds. For every $\varepsilon>0$ and every positive integer $M$, every graph $G$ of order $n\geq M$ has an $\varepsilon$-regular partition into $K$ parts, where $M\leq K\leq s(M,\varepsilon)$.
\end{theorem}
For $\varepsilon>0$, a partition obtained from the regularity lemma is also called an $\varepsilon$-{\em regular partition}. In the original proof of the regularity lemma, the bound $s(m,\varepsilon)$ on the number of parts is a tower of twos\footnote{We define the \emph{tower function} $T(n)$ of \emph{height} $n$ as follows: $T(1)=2$, and $T(n)=2^{T(n-1)}$ for $n\geq2$.} of height $O(1/\varepsilon^2)$. Unfortunately, this is not far from the truth. Gowers \cite{G} showed that such an enormous bound is indeed necessary. In this paper, we will always assume that $m\gg 1/\varepsilon$, so $K\gg1/\varepsilon$ as well.

Recall that the \emph{cut metric} $\d$ between two (edge-weighted) graphs $G$ and $H$ on the same vertex set $V$ is defined by
\begin{equation}\label{cut1}
\d(G,H)=\max_{U,W\subseteq V}\frac{|\e_G(U,W)-\e_H(U,W)|}{|V|^2}.
\end{equation}
Here $\e_G(U,W)$ denotes the total weight of the edges with one endvertex in $U$ and the other endvertex in $W$. 

When $G$ and $H$ are bipartite graphs with bipartition $V_1\sqcup V_2$, we define the cut metric as
\begin{equation}\label{cut2}
\d(G,H) = \max_{U\subseteq V_1, W\subseteq V_2}\frac{|\e_G(U,W)-\e_H(U,W)|}{|V_1||V_2|}.
\end{equation}
If $|V_1|\sim|V_2|$ is large, definitions (\ref{cut1}) and (\ref{cut2}) differ only by a small constant factor.

The cut distance is ``large", i.e.\ $\d(G,H) > \varepsilon$, if and only if there exists subsets $X\subseteq V_1, Y\subseteq V_2$ with $|\e_G(X,Y)-\e_H(X,Y)| > \varepsilon |V_1||V_2|$. Let us observe that in such a case we have $|X|\cdot|Y| > \varepsilon |V_1|\cdot|V_2|$. In particular, $|X| > \varepsilon |V_1|$ and $|Y| > \varepsilon |V_2|$.

If $G$ and $H$ are graphs of the same order and $\sigma:V(H)\to V(G)$ is a bijection, we consider the graph $H^\sigma$ isomorphic to $H$ whose vertex-set is $V(G)$ and edges are $\{\sigma(e)\mid e\in E(H)\}$. Then we define the cut distance $\d(G,H)$ as the minimum of $\d(G,H^\sigma)$ taken over all bijections $\sigma:V(H)\to V(G)$.

By using the language of cut metrics, a weaker condition than that of the regular partition $\mathcal{P}=\{V_1,\dots,V_K\}$ is that, for all but at most $\varepsilon K^2$ pairs ($V_i,V_j$), we have $\d(G[V_i\sqcup V_j], K(V_i,V_j;p_{ij}))<\varepsilon$, where $K(V_i,V_j;p_{ij})$ is the complete bipartite graph defined on the vertex set $V_i\sqcup V_j$ with all edges having weight $p_{ij}=d_G(V_i,V_j)$.

The cut distance gives us a way to describe the similarity between two graphs, and it is widely used in graph limit theory \cite{graphlimit}. Another widely used way to describe the similarity between two large graphs is comparing the homomorphism densities of small graphs. Let $\hom(F,G)$ denote the number of homomorphisms of $F$ into $G$, i.e. the number of mappings $\phi: V(F)\to V(G)$ such that for each $xy\in E(F)$, $\phi(x)\phi(y)\in E(G)$. Then we define the {\em homomorphism density}:
\[
t(F,G)=\frac{\hom(F,G)}{|V(G)|^{|V(F)|}}.
\]

To compare the cut distance and homomorphism densities, we have the following fundamental relation. For the more general version on graphons, see \cite{graphlimit, LS}.
\begin{lemma}[Counting Lemma]\label{lem:counting}
Let $G$ and $G^\prime$ be two graphs defined on the same vertex set. Then for any graph $F$,
\[
|t(F,G)-t(F,G^\prime)|\leq|E(F)|\, \d(G,G^\prime).
\]
\end{lemma}

\subsection{Quasirandomness}
\label{sec:quasirandomness}

Quasirandom graphs are graphs that share many properties with random graphs. The definition of quasirandomness was first introduced in a seminal paper by Chung, Graham and Wilson \cite{quasi}. In that paper, they listed several equivalent definitions of quasirandom graphs, but essentially, quasirandom graphs are graphs close to random graphs in the sense of the cut distance. 

We will introduce a general form of quasirandomness. In order to avoid to use probability, we define the corresponding edge-weighted (complete) graphs.

We will focus on a more general setting of quasirandom graphs. Let $P$ be an $m\times m$ symmetric matrix with non-negative entries $p_{ij}$ ($i,j\in [m]$), where $0\le p_{ij} \le 1$. Let $K(n^{(m)},P)$ be the complete edge-weighted graph, which is defined on the vertex set $V_1\sqcup\dots\sqcup V_m$, and for every $i\in[m]$ we have $|V_i|=n$ and the weight of edges between $V_i$ and $V_j$ is given by the $(i,j)$-entry $p_{ij}$ of $P$. Although we can use the same method to deal with more general graphs, for the convenience we only consider the case that each part of the graph has the same size and we will assume that the diagonal of $P$ is $0$ if $m\ge 2$. Then $K(n^{(m)},P)$ is actually the complete $m$-partite graph. An exception to the latter assumption is for $m=1$, when $p_{11}$ is nonzero.

Recall \cite{graphlimit} that a {\em graphon} is a symmetric measurable function $W:[0,1]^2\to [0,1]$. Let $\widetilde{K}_m=K_m(P)$ be the {\em quotient graph} of $K(n^{(m)},P)$, that is, $\widetilde{K}_m$ has $m$ vertices, and the weight of the edge $ij$ is $p_{ij}$. Let $W_m=W_m(P): [0,1]^2 \to [0,1]$ be the {\em step function} of $\widetilde{K}_m$, that means $\inner{W_m}{(\frac{i-1}{m},\frac{i}{m}]\times(\frac{j-1}{m},\frac{j}{m}]}=p_{ij}$. We say that a sequence of graphs $\{G_n\}$ is {\em $W_m$-quasirandom} if $G_n\to W_m$ as $n\to\infty$ where the convergence is in the cut-distance metric, i.e.
$$
   \d(G_n, \widetilde{K}_m) \to 0, \quad \hbox{ as } n\to \infty.
$$
Given a graph $G$ of order $mn$, we say $G$ is {\em $\varepsilon$-$W_m$-quasirandom} if $\d(G,K(n^{(m)},P))<\varepsilon$, and we write $G\in\Q(n^{(m)},P,\varepsilon)$ in such a case. If $m=1$ and we have $P=[p]$ with $0\le p\le 1$, we write just $\Q(n,p,\varepsilon)$ and we observe that this is precisely the set of quasirandom graphs of Chung et al.~\cite{quasi}.

If $G\in \Q(n^{(m)},P,\varepsilon)$, we may assume that the vertex-set of $G$ is the same as the vertex-set of $K(n^{(m)},P)$. This means that $V(G)$ is partitioned as $V_1,\dots,V_m$, where $|V_i| = n$ for every $i\in[m]$, and that $\d(G,K(n^{(m)},P)) \le \varepsilon$ under the definition (\ref{cut1}).

\begin{lemma}\label{lem:QR B_ij and T_ijk}
Let $G\in\Q(n^{(m)},P,\varepsilon)$, and let $V_1\sqcup\cdots\sqcup V_m$ be an equitable vertex-partition of $V(G)$ corresponding to the vertex partition in $K(n^{(m)},P)$. For any distinct $i,j,k\in [m]$, let $B_{ij}$ be the complete bipartite graph on $V_i\sqcup V_j$ with all edges of weight $p_{ij}$, and let $T_{ijk}$ be the complete tripartite graph $B_{ij} \cup B_{jk} \cup B_{ik}$ on the vertices $V_i\sqcup V_j \sqcup V_k$. Then we have
\[
\d(G[V_i,V_j],B_{ij}) \le \tfrac{1}{4}m^2\varepsilon \quad \hbox{and} \quad \d(G[V_i,V_j,V_k],T_{ijk}) \le \tfrac{1}{9} m^2\varepsilon.
\]
\end{lemma}

\begin{proof}
We will prove the bound for $B_{ij}$. The proof for $T_{ijk}$ is similar. Let $V=V(G)$. Note that $|V|=nm$. For every $X\subseteq V_i$ and $Y\subseteq V_j$, using the definition (\ref{cut1}) of the cut distance, we have
\[
\Big|\e_G(X,Y)-|X||Y|\,p_{ij}\Big| \le \d(G,K(n^{(m)},P))|V|^2 \le 
\varepsilon |V|^2 = \varepsilon m^2n^2.
\]
By using the fact that $B_{ij}$ has $2n$ vertices, this shows that $\d(G[V_i,V_j],B_{ij}) \le m^2\varepsilon/4$.
\end{proof}

Given a graph $G$ of order $N$ and a partition $\mathcal{P}=\{V_1,V_2,\dots,V_m\}$, we say $\mathcal{P}$ is an {\em $\varepsilon$-quasirandom partition} if for every $i\in[m]$ we have $|V_i|=N/m$, and for every $1\leq i<j\leq m$, $\d\big(G[V_i\cup V_j],K((N/m)^{(2)},d(V_i,V_j))\big)<\varepsilon$. Clearly, one can obtain an $\varepsilon$-quasirandom partition from an $\varepsilon$-regular partition $\mathcal{P}^\prime=\{V_1,\dots,V_m\}$. By removing all the edges between irregular pairs, we obtain an $\varepsilon$-quasirandom partition of the resulting graph.

A regular partition of a large graph can be obtained by means of an efficient algorithm. See, for example, \cite{N94}. Recently Tao \cite{Tao10} provided a probabilistic algorithm which produces an $\varepsilon$-regular partition with high probability in constant time (depending on $\varepsilon$). Frieze and Kannan \cite{FK} found a quadratic-time algorithm for constructing an $\varepsilon$-quasirandom partition. In this paper, we will use a more recent deterministic PTAS due to Fox et al.~\cite{FLZ}.

\begin{theorem}[\cite{FLZ}]\label{thm:algpartition}
There exists an algorithm, which, given $\varepsilon>0$, and $0<\alpha<1$, an integer $m$, and a graph $G$ on $n$ vertices that admits an $\varepsilon$-regular partition with $m$ parts, outputs a $(1+\alpha)\varepsilon$-regular partition of $G$ into $m$ parts in time $O_{\varepsilon,\alpha,m}(n^2)$.
\end{theorem}

\subsection{Notation}
\label{sec:notation}

We will use standard graph theory terminology and notation \cite{Diestel}, as well as that of topological graph theory and graph limit theory as given in \cite{top} and \cite{graphlimit}, respectively. We will also use the following notation: $A(n)\sim B(n)$ means $\lim_{n\to\infty}A(n)/B(n)=1$, and $A(n)\ll B(n)$ means $\lim_{n\to\infty}A(n)/B(n)=0$. By $X\sqcup Y$ we denote the disjoint union of $X$ and $Y$.
%, and we set $X\oplus Y = (X\times Y )\sqcup(Y\times X)$. 
We say that an event $E(n)$ happens \emph{asymptotically almost surely} (abbreviated \emph{a.a.s.}) if $\P(E(n))\to1$ as $n\to\infty$.

Given a graph $G$, suppose $X,Y\subseteq V(G)$. We use $E(X,Y)$ to denote the set of edges between $X$ and $Y$, and $\e(X,Y)=|E(X,Y)|$. We also use $\e(G)$ to denote $|E(G)|$. For every $u,v\in V(G)$, let $\n(u,v)$ be the number of common neighbors of $u$ and $v$, and let $\p(u,v)$ be the number of paths of length $3$ between $u$ and $v$. We use $[n]$ to denote the set of integers $\{1,2,\dots,n\}$.

%%%%%%%%%%%%%%%%%%%%%%%%%%%%%%%%%%%%%%%%%%%%%

\section{Genus of multipartite quasirandom graphs}%

Throughout this section, we assume that our graph is sufficiently large and that $\varepsilon$ is small. The section is divided into three subsections. In the first subsection, we will approximate the genus of a quasirandom graph, which is a graph with a small cut distance to $K(n,p)$, the complete graph with all edges weighted by $p$. In the second subsection, we will approximate the genus of a bipartite quasirandom graph. In the final subsection, we will use the results from the first two subsections to approximate the genus of a multipartite quasirandom graph.

\subsection{Genus of quasirandom graphs}%

Given $\varepsilon>0$, we consider a graph $G\in\Q(n,p,\varepsilon)$. Recall that this means that $\d(G,K(n,p))<\varepsilon$. We consider $0<p\le 1$ to be a constant. Along the way, we will need that %$\varepsilon < 1/12$ and later also that 
$\varepsilon < (p/4)^{12}$, which we will assume later in this section. We have the following result.

\begin{lemma}\label{lem:codegree}
Suppose $\varepsilon>0$ and $0\leq p\leq 1$, and let $G\in\Q(n,p,\varepsilon)$ be a simple graph. Then we have
\begin{equation}
2\sum_{uv\in E(G)}|\n(u,v)-p^2n|\leq\sqrt{12\varepsilon}\, n^3.
\label{eq:4}
\end{equation}
\end{lemma}

\begin{proof}
Since $G\in\Q(n,p,\varepsilon)$, Counting Lemma \ref{lem:counting} yields the following inequalities:
\begin{equation}
  | \hom(K_2,G) - pn^2 | \le \varepsilon n^2 \quad \hbox{and} \quad
  | \hom(K_3,G) - p^3n^3 | \le 3\varepsilon n^3.
\label{eq:5}
\end{equation}
Thus, we have
\begin{equation}
   \e(G) = \tfrac{1}{2} \hom(K_2,G) \leq \tfrac{1}{2}(p+\varepsilon)n^2
\label{eq:6}
\end{equation}
and
\begin{equation}
   2\sum_{uv\in E(G)}\n(u,v) = \hom(K_3,G) \geq (p^3-3\varepsilon)n^3.
\label{eq:7}
\end{equation}

Let $K_4^-$ be the graph obtained from $K_4$ by deleting an edge and observe that 
\begin{equation}
   2\sum_{uv\in E(G)}\n^2(u,v) = \hom(K_4^-,G) \leq (p^5 + 5\varepsilon)n^4.
\label{eq:8}
\end{equation}
By using the Cauchy-Schwarz inequality and applying (\ref{eq:6})--(\ref{eq:8}), we now derive:
\begin{align*}
\Bigl(2\sum_{uv\in E(G)}|\n(u,v) - p^2n|\Bigr)^2 
%=\, &\bigg(\sum_{u,v\in V(G)}|\n(u,v)-p^2n|a_{uv}\bigg)^2\\
 \leq ~ & 4\e(G) \sum_{uv \in E(G)} \big|\n(u,v)-p^2n\big|^2\\
 \leq ~ & 2n^2\Bigl(\sum_{uv \in E(G)}\n^2(u,v) - 2p^2n\sum_{uv \in E(G)}\n(u,v) + p^4n^2\e(G)\Bigr)\\[1mm]
% \sim\, & n^2\big(\hom(K_4^-,G)-4p^2n\sum_{uv\in E(G)}\n(u,v)+2p^4n^2\e(G)\big)\\
 \leq ~ & n^6(p^5+5\varepsilon -2p^5 + 6\varepsilon p^2 + p^5 + \varepsilon p^4) \\[2mm]
 = ~ & 12\varepsilon\, n^6,
\end{align*}
and this proves the lemma.
\end{proof}

Given an $\varepsilon$-quasirandom graph $G\in\Q(n,p,\varepsilon)$, we choose a real number $t=t(n)$ such that $n^{-1/2}\ll p/t\ll n^{-(1-\varepsilon)/(2-\varepsilon)}$ (we will use this upper bound later to limit the number of short blossoms in the graph). Let $D\in\mathcal{D}(G)$ be the corresponding digraph of $G$, and we partition its edges into $t$ parts, putting each edge in one of the parts uniformly at random. We call the resulting digraphs $D_1,\dots,D_t$. Now for each $D_i$, let $\H_i$ be the $3$-uniform hypergraph where $V(\H_i)$ is the edge set of $D_i$, and $E(\H_i)$ is the set of directed triangles in $D_i$. For convenience, we write $p_1 = p/t$.

\begin{lemma}\label{lem:quasi1}
Let\/ $p,\varepsilon,t$ and $p_1$ be as above and let $\Delta=np_1^2/4$. Then there exists a real number $\delta = \Theta(\varepsilon^{1/4})$ such that for each $i\in [t]$, we have
\[
|\{x\in V(\H_i)\mid(1-\delta)\Delta\leq\deg(x)\leq(1+\delta)\Delta\}|\geq(1-\delta)|V(\H_i)| \quad \hbox{a.a.s.}
\]
\end{lemma}

\begin{proof}
We first return to the graph $G$. 
Let $\lambda^2=\frac{\sqrt{13\varepsilon}}{p^2(p-2\varepsilon)}$. For every edge $uv\in E(G)$, we say $uv$ is {\em unbalanced} if there are at least $(1+\lambda)np^2$ paths of length $2$ between $u$ and $v$, or at most $(1-\lambda)np^2$ paths of length $2$ between them. If there are at least $\lambda\e(G)$ unbalanced edges, then
\[
2\sum_{uv\in E(G)}|\n(u,v)-np^2| \geq 2\lambda^2\e(G)np^2
> \lambda^2(n^2p-2\varepsilon n^2)np^2 > \sqrt{12\varepsilon}n^3.
\]
This contradicts (\ref{eq:4}) and proves that the number of unbalanced edges is smaller than $\lambda\e(G)$.

We say an edge of $G$ is {\em balanced} if it is not unbalanced and let $\B$ denote the set of all balanced edges in $G$. As shown above, $|\B| \ge (1-\lambda)\e(G)$. Since we create $D_i$ by selecting edges uniformly at random from $D$, we have
\[
\E(|\B\cap E(D_i)|)=\frac{|\B|}{t} \quad \hbox{and} \quad
\E(|\B\cap E(D_i)|^2)=\frac{|\B|^2-|\B|}{t^2}.
\]
By Chebyshev's inequality,
\begin{equation}\label{cheb}
\P\Big(\big||\B\cap E(D_i)|-\E(|\B\cap E(D_i)|)\big|>\varepsilon\frac{|\B|}{t} \Big)\leq\frac{|\B|t^2}{t^2\varepsilon^2|\B|^2}=\frac{1}{\varepsilon^2|\B|}=o(1).
\end{equation}
This means a.a.s. that $D_i$ contains at least $(1-\varepsilon)|\B|/t$ edges which are balanced in $G$. In the digraph $D_i$, for every $e\in E(D_i)$, let $\mathcal{T}_i(e)$ be the set of directed triangles which contain $e$. Let $\mathcal{T}(e)$ be the set of (undirected) triangles in the graph $G$ which contain $e$. Let us now consider a balanced edge $e$ of $G$ that belongs to $D_i$. Then
\[
\E(|\mathcal{T}_i(e)|)=\frac{\mathcal{T}_G(e)}{4t^2},\qquad \E(|\mathcal{T}_i(e)|^2)=\frac{\mathcal{T}^2_G(e)-\mathcal{T}_G(e)}{16t^4}.
\]
By Chebyshev's inequality,
\[
\E\big(\big||\T_i(e)|-\E(|\T_i(e)|)\big|>\varepsilon\E(|\T_i(e)|)\big)\leq\frac{\E^2(|\T_i(e)|)-\E(|\T_i(e)^2|)}{\varepsilon^2\E^2(|\T_i(e)|)}=\frac{1}{|\T_G(e)|}=o(1).
\]
This means at least $(1-\varepsilon)^2|\B|/t$ edges in $D_i$ are contained in at least $(1-\varepsilon)(1-\lambda)\Delta$ directed triangles and in at most $(1+\varepsilon)(1+\lambda)\Delta$ triangles.

Since $|\B|/t\geq(1-\lambda)\e(G)/t$, and a.a.s. $|V(\H_i)|<(1+\varepsilon)\e(G)/t$ (also by Chebyshev's inequality), we choose $\delta\geq\delta_0 := \max\{\lambda+\varepsilon+\varepsilon\lambda, \psi(\varepsilon,\lambda)\}$, where $\psi(\varepsilon,\lambda) = \frac{\lambda(1-\varepsilon)^2+\varepsilon(3-\varepsilon)}{1+\varepsilon}$. Therefore,
\[
\begin{split}
&(1+\varepsilon)(1+\lambda)\Delta = (1+\lambda+\varepsilon+\varepsilon\lambda)\Delta<(1+\delta)\Delta,\\
&(1-\varepsilon)(1-\lambda)\Delta = (1-\lambda-\varepsilon+\varepsilon\lambda)\Delta>(1-\delta)\Delta,\\
&(1-\varepsilon)^2\frac{|\B|}{t}\geq(1-\varepsilon)^2(1-\lambda)\frac{|V(\H_i)|}{1+\varepsilon} = (1-\psi(\varepsilon,\lambda))|V(\H_i)|\geq(1-\delta)|V(\H_i)|,
\end{split}
\]
which completes the proof.
\end{proof}

Now we fix the value $\Delta$ in Lemma \ref{lem:quasi1}. The lemma shows that a.a.s. $\H_i$ satisfies condition (1) of Theorem \ref{thm:hypergraphmatching}. Condition (2) of Theorem \ref{thm:hypergraphmatching} holds trivially since for every $a,b\in V(\H_i)$, at most one triangle of $G$ contains both $a$ and $b$ as its edges. Condition (3) holds by the following lemma.

\begin{lemma}\label{lem:c3}
Let $p,\varepsilon,t,\Delta$ and $\delta$ be as in Lemma \ref{lem:quasi1} and let $i\in [t]$. Suppose, moreover, that $\varepsilon \le (p/4)^{12}$. Let $P^\delta$ be the set of pairs of vertices $(u,v)\in V^2(D_i)$ such that the number of directed paths from $v$ to $u$ of length $2$ in $D_i$ is at least $(1+\delta)\Delta$ and let
$F_i$ be the number of directed triangles in $D_i$ that contain at least one directed edge $\overrightarrow{uv}\in P^\delta$. Then a.a.s. $F_i<\delta\Delta|V(\H_i)|$.
\end{lemma}

\begin{proof}
Let $\lambda$ be as in Lemma~\ref{lem:quasi1}. We may assume that $\delta > 2\lambda$. 

By using Chebyshev's inequality, we can pass the counting properties from $G$ to $D_i$ (within the factor $1 + \varepsilon$). Thus it suffices to consider the graph $G$. Let $F$ be the number of triangles in $G$ that contain at least one unbalanced edge in $G$. Since $D_i$ is constructed randomly, if $F<(1-\varepsilon)\delta np^2\e(G)$, then $F_i<\delta\Delta|V(\H_i)|$ a.a.s.

By Lemma \ref{lem:codegree} and Lemma~\ref{lem:quasi1}, we have
\[
2\sum_{uv\in E(G)\setminus \B}\n(u,v)<\sqrt{12\varepsilon}\,n^3 + 2|E(G)\setminus \B|\, p^2n\leq \sqrt{12\varepsilon}\,n^3 + 2\lambda np^2 \e(G).
\]
From the earlier proofs we have that $\delta < 10\varepsilon^{1/4}$. By using the assumption that $\varepsilon < (p/4)^{12}$, we have
\[
\sqrt{12\varepsilon} < (\tfrac{1}{2}-\varepsilon)\delta p^2(1-\varepsilon)\frac{p}{2}.
\]
This implies that the following inequalities hold (a.a.s.):
\begin{align*}
F&<\sum_{uv\in E(G)\setminus\B}\n(u,v)<\lambda np^2\e(G)+\sqrt{12\varepsilon}\, n^3/2\\
&< \tfrac{1}{2}\delta np^2\e(G) + (\tfrac{1}{2}-\varepsilon)\delta np^2\e(G)
= (1-\varepsilon)\delta np^2\e(G),
\end{align*}
which completes the proof.
\end{proof}

Now we are going to apply Theorem \ref{thm:hypergraphmatching} to construct a large set of blossom-free triangles. Using them together with Lemma \ref{lem:rotation}, we will be able to obtain a near triangular embedding. Specifically, we say that an embedding of a graph is \emph{$\varepsilon$-near triangular} (\emph{$\varepsilon$-near quadrangular}) if all but at most $2\varepsilon \e(G)$ edges of $G$ belong to two triangular (quadrangular) faces.

\begin{theorem}\label{thm:quasiembed}
Let $0<p\le 1$ be a constant and $0<\varepsilon < (p/4)^{12}$. Then there is a constant $n_0$ such that for every natural number $n\ge n_0$, every graph $G\in\Q(n,p,\varepsilon)$ admits a $9\varepsilon$-near triangular embedding.
\end{theorem}

\begin{proof}
Suppose that $G\in\Q(n,p,\varepsilon)$. We will use the notation introduced earlier in this section. We have that for every $i\in[t]$, $\H_i$ satisfies conditions (1)--(3) in Theorem \ref{thm:hypergraphmatching} a.a.s. That means that if $q$ is the probability that $\H_i$ does not satisfy all of the conditions, then $q\to 0$ as $n\to \infty$. Let $J\subseteq [t]$ be the index set such that for every $j\in J$, $\H_j$ satisfies all three conditions listed in Theorem \ref{thm:hypergraphmatching}. Let $I = [t]\setminus J$. By Markov's inequality we obtain $|I|\leq\varepsilon t$ a.a.s. Note that ``a.a.s.'' comes from the way we construct $D_i$, which means that there exist $D_1,\dots,D_t$ such that $|I|\leq \varepsilon t$. 

For every $i\in J$, apply Theorem \ref{thm:hypergraphmatching}. For sufficiently large $n$, there exists a matching $M_i$ of $\H_i$ of size at least $(1-\varepsilon)\frac{|E(D_i)|}{3}$. Let $\H_i^{-1}$ be the $3$-uniform hypergraph defined on the digraph $D^{-1}$, then $M_i^{-1}$ is a large matching in $\H_i^{-1}$. Applying Theorem \ref{thm:hypergraphmatching} again, we obtain another matching $M_i^\prime$ in $\H_i^{-1}$ that is disjoint with $M_i^{-1}$, and has a size of at least $(1-\varepsilon)\frac{|E(D_i)|}{3}$. Therefore, in the graph $D_i\cup D_i^{-1}$, $M_i\cup M_i^\prime$ has size at least $2(1-\varepsilon)\frac{|E(D_i)|}{3}$, and does not have non-simple blossoms of length $2$.

In order to apply Lemma \ref{lem:rotation} to obtain a rotation system and a surface embedding, we have to remove one of the cycles from each of the blossoms that appear in $M_i\cup M_i^\prime$. We first consider short blossoms, that is, blossoms of length at most $1/\varepsilon$. We use $\overrightarrow{\mathcal{B}_j}$ to denote a digraph of a $3$-blossom of length $j$ in $D\cup D^{-1}$, and $\mathcal{B}_j$ to denote its underlying simple graph in $G$, where $j$ is an integer such that $2\leq j\leq 1/\varepsilon$. 
Observe that each such $\mathcal{B}_j$ is a subgraph of $G$ that is isomorphic to the wheel graph of length $j$ (the cycle of length $j$ together with the center vertex adjacent to all vertices on the cycle). By the Counting Lemma \ref{lem:counting}, we have
\[
\hom(\mathcal{B}_j,G)\leq n^{j+1}p^{2j}+2j\varepsilon n^{j+1}.
\]

By the way, we construct $D_i$, for every $i\in [t]$, given a $3$-blossom simple graph $\mathcal{B}_j$ in $G$, $\P(\overrightarrow{\mathcal{B}_j}\in D_i\cup D_i^{-1})=\frac{1}{2^{j-1}t^{2j}}$. Let $N_i(j)$ denote the number of $\overrightarrow{\mathcal{B}_j}$ in $D_i\cup D_i^{-1}$. Then by Chebyshev's inequality
\[
N_i(j)<(1+\varepsilon)\Big(\frac{n^{j+1}p_1^{2j}}{2^{j-1}}+\frac{2j\varepsilon n^{j+1}p_1^{2j}}{p^{2j}2^{j-1}}\Big)\ll n^2p_1, \quad a.a.s.
\]
That means $\sum_{j=1}^{1/\varepsilon}N_i(j)<\varepsilon(1-\varepsilon) n^2p_1/2<\varepsilon |E(D_i)|$, a.a.s. Then we can remove a triangle from each blossom of length at most $1/\varepsilon$, we obtain a subset of $M_i\cup M_i^\prime$ of size at least $2(1-3\varepsilon)\frac{|E(D_i)|}{3}$.

In the next step we will argue that there is only a small number of long blossoms. If $\mathcal{B}$ and $\mathcal{B}^\prime$ are two blossoms with center $v$ in $M_i\cup M_i^\prime$, it is clear that the tips of $\mathcal{B}$ and $\mathcal{B}^\prime$ are disjoint. Therefore if $v$ has $l$ neighbours in $D_i$, at most $\varepsilon l$ different $3$-blossoms of length at least $1/\varepsilon$ have center $v$. Therefore, the total number of long blossoms is at most $2\varepsilon|E(D_i)|$. By removing one of the triangles from each long blossom, we finally obtain a blossom-free subset $\mathfrak{M}_i$ of $M_i\cup M_i^{-1}$, such that
\[
|\mathfrak{M}_i|\geq2(1-3\varepsilon)\frac{|E(D_i)|}{3}-2\varepsilon|E(D_i)|=2(1-6\varepsilon)\frac{|E(D_i)|}{3}.
\]

Note that for every $i\in J$, by Chebyshev's inequality, there exists a set $J^\prime\subseteq J$ such that $|J^\prime|\geq(1-\varepsilon)|J|$ and for every $i\in J^\prime$,
\[
(1-\varepsilon)\frac{\e(G)}{t}\leq |E(D_i)|\leq(1+\varepsilon)\frac{\e(G)}{t}.
\]

Now we take the union of all $\mathfrak{M}_i$ with $i\in J^\prime$, the union set is blossom-free since all pairs of the graphs $D_i$ and $D_{i'}$ with $i\neq i'$ are edge-disjoint. Now, applying Lemma \ref{lem:rotation}, we obtain an embedding $\Pi$ from $\bigcup_{i\in J^\prime}\mathfrak{M}_i$ such that
\[
\begin{split}
 3 f_3(\Pi) &\geq 3\sum_{i\in J^\prime}|\mathfrak{M}_i|
 \geq 2(1-6\varepsilon)(1-\varepsilon)\frac{\e(G)}{t} |J'| \\
 &\geq 2(1-9\varepsilon)\e(G),
\end{split}
\]
which completes the proof.
\end{proof}

\begin{theorem}\label{thm:quasigenus}
Let $0<p\le 1$ be a constant and $0<\varepsilon < (p/4)^{12}$. Then there is a constant $n_0$ such that for every natural number $n\ge n_0$, the genus and the non-orientable genus of every graph
$G\in\Q(n,p,\varepsilon/18)$ lie within the following bounds:
\[
\frac{\e(G)}{6}\leq \g(G)\leq(1+\varepsilon)\frac{\e(G)}{6}
\]
and
\[
\frac{\e(G)}{3}\leq \ng(G)\leq(1+\varepsilon)\frac{\e(G)}{3}.
\]
\end{theorem}

\begin{proof}
By Theorem \ref{thm:quasiembed}, $G$ has an $\varepsilon/2$-near triangular embedding $\Pi$. Therefore,
\[
\begin{split}
\g(G) &\leq \g(G,\Pi) \le \tfrac{1}{2}(\e(G)-f(\Pi)) \le \tfrac{1}{2}(\e(G)-f_3(\Pi))\\
      &\leq \tfrac{1}{2}(\e(G)-\tfrac{2}{3}(1-\tfrac{\varepsilon}{2})\e(G))\leq(1+\varepsilon)\frac{\e(G)}{6}.
\end{split}
\]

The lower bound for $\g(G)$ follows by \cite[Proposition 4.4.4]{top} directly. For the non-orientable genus, we use essentially the same proof together with the non-orientable version of Lemma \ref{lem:rotation} to obtain a non-orientable $9\varepsilon$-near-triangular embedding in the proof of Theorem \ref{thm:quasiembed}.
\end{proof}

\subsection{Bipartite and tripartite quasirandom graphs}%%%%%

Given $\varepsilon>0$, we now consider a graph $G\in\Q(n^{(2)},P,\varepsilon)$, 
where 
$P = 
\begin{bmatrix}
0 & p\\
p & 0
\end{bmatrix}$. 
For this case, we simply write $G\in\Q(n^{(2)},p,\varepsilon)$.
This means $G$ is defined on the vertex set $V_1\sqcup V_2$, each $V_i$ of size $n$, and $\d(G,K(n^{(2)},p))<\varepsilon$. 

Recall that for $u,v\in V(G)$, $\p(u,v)$ denotes the number of $(u,v)$-paths of length 3. We also define $\w(u,v)$ as the number of $(u,v)$-walks of length 3.
These two quantities are closely related:
$$
\w(u,v) = \left\{
  \begin{array}{ll}
     \p(u,v), & \hbox{if $uv\notin E(G)$;} \\
     \p(u,v) + \deg_G(u) + \deg_G(v) -1, & \hbox{if $uv\in E(G)$.}
  \end{array}
\right.
$$
This implies that 
\begin{eqnarray}
  \hom(C_4,G) &=& 2\sum_{uv\in E(G)}\w(u,v) \nonumber \\[2mm]
  &=& 2\sum_{uv\in E(G)}\p(u,v) + 2\sum_{v\in V(G)}\deg^2(v) - 2\e(G) \nonumber \\[2mm]
  &=& 2\sum_{uv\in E(G)}\p(u,v) + 2\hom(P_3,G) - 2\e(G). \label{eq:homC4}
\end{eqnarray}

The following lemma is an analogous result to Lemma \ref{lem:codegree}.

\begin{lemma}\label{lem:bi1}
Let $0 < \varepsilon < 1$, $n\ge 13\varepsilon^{-1}$, $0<p<1$, and let $G\in\Q(n^{(2)},p,\varepsilon)$. Then 
\[
2\sum_{uv\in E(G)}|\p(u,v)-n^2p^3|\leq\sqrt{17\varepsilon}\, n^4.
\]
\end{lemma}

\begin{proof}
First, by the Counting Lemma \ref{lem:counting}, we have $2\e(G) = \hom(K_2,G) \leq (p+\varepsilon)n^2$, $\hom(P_3,G) \leq (p^2+2\varepsilon)n^3$, and $\hom(C_4,G) \ge (p^4-4\varepsilon)n^4$.
By using (\ref{eq:homC4}) we obtain
\[
   2\sum_{uv\in E(G)}\p(u,v) \geq (p^4-4\varepsilon)n^4 - 2(p^2+2\varepsilon)n^3 + (p-\varepsilon)n^2.
\]
Let $Q$ be the graph obtained from the 6-cycle by adding one ``long diagonal". Again, we have:
\[
   2\sum_{uv\in E(G)}\w^2(u,v) = \hom(Q,G) \le (p^7+7\varepsilon)n^6.
\]
Therefore, by the Cauchy-Schwarz inequality,
\begin{align*}
 \Bigl( 2 \sum_{uv\in E(G)}| & \p(u,v)-n^2p^3| \Bigr)^2 \leq 2n^2\sum_{uv\in E(G)}\big|\p(u,v) - n^2p^3\big|^2\\
&\leq 2n^2\Bigl( \sum_{uv\in E(G)}\w^2(u,v) - 2n^2p^3\sum_{uv\in E(G)}\p(u,v) +n^4p^6\e(G) \Bigr)\\
&\leq n^2( (p^7+7\varepsilon)n^6 - 2n^2p^3((p^4-4\varepsilon)n^4 - 2(p^2+2\varepsilon)n^3 + (p-\varepsilon)n^2) + (p+\varepsilon)n^6p^6 )\\[1mm]
&< 16\varepsilon n^8 + 13 n^7 \le 17\varepsilon n^8.
\end{align*}
This proves the lemma.
\end{proof}

In order to bound the number of short blossoms, the graphs should not be too dense. The resolution is to split the graph into several quasirandom subgraphs by proceeding in a similar way as in the proof of Theorem \ref{thm:quasiembed}. We choose an integer $t=t(n)$ and let $p_1=p/t$, such that $n^{-\frac{2}{3}}\ll p_1\ll n^{-\frac{2-\varepsilon}{3-\varepsilon}}$. Let $D\in \mathcal{D}(G)$ be the corresponding digraph of $G$. We partition its edges into $t$ parts uniformly at random. The resulting edge-disjoint digraphs are denoted by $D_1,\dots,D_t$. For each $D_i$, let $\H_i$ be the $4$-uniform hypergraph such that $V(\H_i)=E(D_i)$ and the edge-set of $\H_i$ is the set of directed 4-cycles in $D_i$. Now we are going to check condition (1) of Theorem \ref{thm:hypergraphmatching}.

\begin{lemma}\label{lem:3.7}
Let $p,\varepsilon,t$ and $p_1$ be as above, and let $\Delta_2=n^2p_1^3/8$. Then there exists a real number $\delta = \Theta(\varepsilon^{1/4})$ such that for each $i\in [t]$, we have
\[
|\{x\in V(\H_i)\mid(1-\delta)\Delta_2\leq\deg(x)\leq(1+\delta)\Delta_2\}|\geq(1-\delta)|V(\H_i)|
\quad \hbox{a.a.s.}
\]
\end{lemma}

\begin{proof}
Recall that $p=tp_1$. Suppose $\lambda^2=\frac{\sqrt{18\varepsilon}}{p^3(p-\varepsilon)}$ and $uv$ is an edge in $G$. We say $uv$ is {\em balanced} if $\p(u,v)$ is at least $(1-\lambda)n^2p^3$ and at most $(1+\lambda)n^2p^3$; otherwise it is {\em unbalanced}. Assume that at least $\lambda\e(G)$ edges are unbalanced. Then we have
\[
2\sum_{uv\in E(G)}|\p(u,v)-n^2p^3|\geq \lambda^2n^2p^3(n^2p-\varepsilon n^2)\geq \sqrt{18\varepsilon}\, n^4,
\]
which contradicts Lemma \ref{lem:bi1}. Thus, at least $(1-\lambda)\e(G)$ edges in $G$ are balanced.

Chebyshev's inequality implies that
\[
(1-\varepsilon)\frac{\e(G)}{t}\leq |E(D_i)|\leq (1+\varepsilon)\frac{\e(G)}{t}
\]
holds a.a.s.\ for every $i\in [t]$.

Let $U$ be the number of balanced edges in $G$. Then a.a.s. $D_i$ contains at least $(1-\varepsilon)U/t$ edges which are balanced in $G$, and most of them are still in $D_i$. To be more precise, at least $(1-\varepsilon)^2U/t$ edges are contained in at least $(1-\varepsilon)(1-\lambda)n^2p_1^3/8$ directed cycles of length $4$, and are contained in at most $(1+\varepsilon)(1+\lambda)n^2p_1^3/8$ directed cycles of length $4$. Now we let $\delta \geq \max\{\lambda+\varepsilon+\varepsilon\lambda,\psi(\varepsilon,\lambda)\}$, where $\psi(\varepsilon,\lambda)$ is given in Lemma \ref{lem:quasi1}. We have a.a.s.
\[
\begin{split}
(1-\varepsilon)^2\frac{U}{t}&\geq(1-\varepsilon)^2(1-\lambda)\frac{\e(G)}{t}\\
&\geq\frac{(1-\varepsilon)^2(1-\lambda)}{1+\varepsilon}|V(\H_i)|\geq (1-\delta)|V(\H_i)|.
\end{split}
\]
We conclude the proof in the same way as for Lemma \ref{lem:quasi1} by observing that 
$(1+\delta)\Delta_2 > (1+\varepsilon)(1+\lambda)\Delta_2$ and $(1-\delta)\Delta_2 < (1-\varepsilon)(1-\lambda)\Delta_2$. This completes the proof.
\end{proof}

It is easy to see that Condition (2) in Theorem \ref{thm:hypergraphmatching} holds a.a.s.\ if $\varepsilon < 1/5$. 
Indeed, if two edges $e$ and $f$ are contained in $7$ different cycles of length four, then $e=vx$ and $f=yv$ have a vertex $v$ in common and there are at least $7$ internally disjoint $(x,y)$-paths of length $2$ joining $x$ and $y$. Let $H$ be the graph that is isomorphic to $K_{2,7}$. Then each such pair $e,f$ determines a copy of $H$ in $D_i$. Let $\kappa$ be the number of copies of $H$ in $D_i$. Since $ p_1\ll n^{-\frac{2-\varepsilon}{3-\varepsilon}}$ and $\varepsilon < 1/5$, we have
\[
\mathbb{E}(\kappa)=O(n^{9}p_1^{14})\ll n^{-\tfrac{1-5\varepsilon}{3-\varepsilon}}=o(1).
\]
Hence by Markov's inequality, we have
\[
\mathbb{P}(\kappa>1)<\mathbb{E}(\kappa)=o(1).
\]
Note that since $p\gg n^{-\frac{2}{3}}$, $\delta\Delta_2\gg1$,  and this implies that Condition (2) in Theorem \ref{thm:hypergraphmatching} holds a.a.s.

In order to verify Condition (3), we have the following lemma. The proof is essentially the same as what we did in Lemma \ref{lem:c3}, and we omit the details.

\begin{lemma}
Let $p,\varepsilon, t, \Delta_2$ and $\delta$ be as in Lemma \ref{lem:3.7} and let $i\in [t]$. Suppose also that $\varepsilon < (p/4)^{16}$. Let $F_i$ be the number of directed cycles of length $4$ in $D_i$ that contain at least one directed edge $\overrightarrow{uv}\in P^\delta$, where $P^\delta$ is the set of pairs of vertices $(u,v)\in V_1\times V_2 \cup V_2\times V_1$ such that the number of directed paths from $v$ to $u$ of length $3$ is at least $(1+\delta)\Delta_2$. Then $F_i<\delta\Delta_2|V(\H_i)|$ a.a.s.
\end{lemma}

These results show that each $\H_i$ satisfies conditions (1)--(3) a.a.s. Then we can apply Theorem \ref{thm:hypergraphmatching} to obtain two almost perfect matchings that are used to construct a rotation system yielding a near-optimal embedding.

\begin{theorem}\label{thm:quasibiembed}
Let $0<p\le 1$ be a constant, and $0<\varepsilon < (p/4)^{16}$. There is a constant $n_0$ such that for every natural number $n\ge n_0$, every graph $G\in\Q(n^{(2)},p,\varepsilon)$ admits a $10\varepsilon$-near quadrangular embedding.
\end{theorem}

\begin{proof}
Similarly as in the proof of Theorem \ref{thm:quasiembed}, we partition the edges of the corresponding digraph $D$ uniformly at random into $t$ subdigraphs $D_1,\dots,D_t$. Then we consider the 4-uniform hypergraphs $\H_i$, $i\in [t]$. The above results show that $\H_i$ satisfy conditions (1)--(3) of Theorem \ref{thm:hypergraphmatching} a.a.s. Let $J\subseteq [t]$ be the set of indices $i$ for which $\H_i$ satisfies these conditions, and let $I = [t]\setminus J$. 
By Markov's inequality, we have $|I|\leq \varepsilon t$.

For every $i\in J$, there exists a matching $M_i$ of $\H_i$ of size at least $(1-\varepsilon)\frac{|E(D_i)|}{4}$, and another matching $M_i^\prime$ in $\H_i^{-1}$ which is disjoint with $M_i^{-1}$, and also has size at least $(1-\varepsilon)\frac{|E(D_i)|}{4}$. Then in the graph $D_i\cup D_i^{-1}$, $M_i\cup M_i^\prime$ has size at least $(1-\varepsilon)\frac{|E(D_i)|}{2}$, and has no non-simple blossoms of length $2$.

Next, we break all blossoms in $M_i\cup M_i^\prime$ by removing one 4-cycle from each of them. For each $j\le 1/\varepsilon$, let $\overrightarrow{\mathcal{B}_j}$ be the digraph of a $4$-blossom of length $j$, and let $\mathcal{B}_j$ be the underlying simple graph. Using the counting lemma, it is easy to see that
\[
\hom(\mathcal{B}_j,G)\leq n^{2j+1}p^{3j} + 3j\varepsilon n^{2j+1}.
\]

For every $4$-blossom simple graph $\mathcal{B}_j$ in $G$ and every $i\in [t]$, the probability that this blossom is in $D_i$ is $\P(\overrightarrow{\mathcal{B}_j}\subseteq D_i\cup D_i^{-1}) = \frac{1}{2^{j-1}t^{3j}}$. Let $N_i(j)$ be the number of $\overrightarrow{\mathcal{B}_j}$ in $D_i\cup D_i^{-1}$. Then we have (by Chebyshev's inequality):
\[
N_i(j)<(1+\varepsilon)\Big(\frac{n^{2j+1}p_1^{3j}}{2^{2j-1}}+\frac{3j\varepsilon n^{2j+1}p_1^{3j}}{p^{3j}2^{2j-1}}\Big)\ll n^2p_1, \quad a.a.s.
\]
This implies that $\sum_{j=1}^{1/\varepsilon}N_i(j) < \varepsilon(1-\varepsilon) n^2p_1 < \varepsilon |E(D_i)|$, a.a.s.

Blossoms in $D_i\cup D_i^{-1}$ centered at $v$ have disjoint tips. Thus, if $v$ has degree $d_v$ in $D_i$, the number of blossoms of length at least $1/\varepsilon$ centered at $v$ is at most $\varepsilon d_v$. Hence, the total number of such long blossoms is at most $\sum_{v\in V(G)} \varepsilon d_v = 2\varepsilon |E(D_i)|$. 

Now we are going to remove one cycle from each of the blossoms. By the above, the total number of $4$-blossoms in $D_i$ is at most $3\varepsilon|E(D_i)|$, hence we obtain a blossom-free subset of $M_i\cup M_i^\prime$ of size at least $(1-7\varepsilon)\frac{|E(D_i)|}{2}$, a.a.s. By Lemma \ref{lem:rotation} there exists a construction of $D_i$ for every $i\in [t]$, and $(1-\varepsilon)^2t$ of them have almost the correct number of edges, and they give rise to a $7\varepsilon$-near quadrangular embedding. Then $G$ has an embedding $\Pi$ such that
\[
4f_4(\Pi)\geq 2(1-7\varepsilon)(1-\varepsilon)^3\e(G)\geq 2(1-10\varepsilon)\e(G),
\]
which completes the proof.
\end{proof}

We are ready to prove a formula for the genus of graphs $G\in \Q(n^{(2)},p,\varepsilon)$. 

\begin{theorem}\label{thm:quasibigenus}
Let $0<p\le 1$, $\varepsilon>0$, and $\varepsilon' = \min\{\varepsilon/10, (p/4)^{16}/10 \}$. There is a number $n_0=n_0(\varepsilon,p)$ such that for every $n\ge n_0$, every positive number $\tau \le \varepsilon'$ and every $G\in\Q(n^{(2)},p,\tau)$, we have:
\[
\frac{\e(G)}{4}\leq\g(G)\leq(1+\varepsilon)\frac{\e(G)}{4},
\]
and
\[
\frac{\e(G)}{2}\leq\ng(G)\leq(1+\varepsilon)\frac{\e(G)}{2}.
\]
\end{theorem}

\begin{proof}
It suffices to consider the orientable genus. (The non-orientable case just uses the non-orientable version of Lemma \ref{lem:rotation}.) The lower bound follows by Euler's formula and the fact that $G$ has no cycles of length 3, see \cite[Proposition 4.4.4]{top}. For the upper bound, we use Theorem \ref{thm:quasibiembed} (with $\varepsilon'$ instead of $\varepsilon$) which gives an $\varepsilon$-near quadrangular embedding $\Pi$. Therefore,
\[
\g(G)\leq\g(G,\Pi)\leq\frac{1}{2}(\e(G)-f_4(\Pi))\leq(1+\varepsilon)\frac{\e(G)}{4}.
\]
This completes the proof.
\end{proof}

% Tripartite

For the case of tripartite quasirandom graphs, we use $\Q(n^{(3)},p,\varepsilon)$ to denote the family of quasirandom graphs defined on the vertex set $V_1\sqcup V_2\sqcup V_3$ with $|V_i|=n$ for every $1\leq i\leq3$, and for every $G\in\Q(n^{(3)},p,\varepsilon)$, we have $\d(G,K(n^{(3)},p))<\varepsilon$. 
In the same way as we were able to use the bipartite version (\ref{cut2}) of the cut distance, we will use the following version in our tripartite case. If $U\subseteq V_1, W\subseteq V_2, Z\subseteq V_3$, we write $\e_G(U,W,Z) = \e_G(U,W)+\e_G(U,Z)+\e_G(W,Z)$ and then define the \emph{cut distance} as
\begin{equation}\label{cut3}
 \d(G,H) = \max_{U\subseteq V_1, W\subseteq V_2, Z\subseteq V_3} \frac{|\e_G(U,W,Z)-\e_H(U,W,Z)|}{|V_1||V_2||V_3|}.
\end{equation}
If $n=|V_1|=|V_2|=|V_3|$ is large, definitions (\ref{cut1}), (\ref{cut2}) and (\ref{cut3}) differ only by a small constant factor.

We start with a result analogous to Lemma \ref{lem:codegree}.

\begin{lemma}\label{lem:tri1}
Let $\varepsilon>0$ and $G\in\Q(n^{(3)},p,\varepsilon)$, where $V(G)=V_1\sqcup V_2\sqcup V_3$. Then for every $i\neq j$ and $u\in V_i$, $v\in V_j$, we have
\[
2\sum_{uv\in E(G)}|\n(u,v)-np^2| \leq \sqrt{12\varepsilon}\, n^3.
\]
\end{lemma}

\begin{proof}
First, we apply the Counting Lemma to obtain the following estimates:
\[
\begin{split}
   & 2\e(G) = \hom(K_2,G) \leq (6p+\varepsilon)n^2, \\[1mm]
   & 2\sum_{uv\in E(G)}\n(u,v) = \hom(K_3,G) \geq (6p^3-3\varepsilon)n^3, \quad \hbox{and}\\[1mm]
   & 2\sum_{uv\in E(G)}\n^2(u,v) = \hom(K_4^-,G) \leq (6p^5+5\varepsilon)n^4.
\end{split}
\]

Now, we use the Cauchy-Schwarz inequality and apply the above to obtain:
\[
\begin{split}
\Bigl(2\sum_{uv\in E(G)}|\n(u,v)-np^2|\Bigr)^2
&\leq 2n^2\sum_{uv\in E(G)}\big|\n(u,v)-np^2\big|^2\\
&\leq n^2\big((6p^5+5\varepsilon)n^4 + 2n^2p^4\e(G) - 2np^2(6p^3-3\varepsilon)n^3\big)\\[1mm]
&\leq (6p^5+5\varepsilon)n^6 + p^4(6p+\varepsilon)n^6 - 2p^2(6p^3-3\varepsilon)n^6\\[1mm]
& = (5\varepsilon+\varepsilon p^4+6\varepsilon p^2)n^6 \leq 12\varepsilon n^6.
\end{split}
\]
This proves the lemma.
\end{proof}

Similarly as before, we choose an integer $t=t(n)$ and let $p_1=p/t$, such that $n^{-\frac{1}{2}}\ll p_1\ll n^{-\frac{1-\varepsilon}{2-\varepsilon}}$. Let $D\in \mathcal{D}(G)$ be the corresponding digraph of $G$. We partition the edge-set of $D$ into $t$ parts uniformly at random, and the resulting edge-disjoint digraphs are denoted by $D_1,\dots,D_t$. For each $D_i$ ($i\in[t]$), we define $\H_i$ as the $3$-uniform hypergraph, whose vertices are the edges in $D_i$, $V(\H_i) = E(D_i)$, and whose edge-set is the set of directed triangles in $D_i$. The following lemma implies that Condition (1) in Theorem \ref{thm:hypergraphmatching} holds a.a.s. in $\H_i$.

\begin{lemma}\label{lem:Delta_3 and delta}
Let $p,\varepsilon,t$ and $p_1$ be as above, and
let $\Delta_3=np_1^2/4$. Then there exists a real number $\delta=\Theta(\varepsilon^{1/4})$ such that for every $i\in[t]$,
\[
  \bigl| \{x\in V(\H_i)\mid(1-\delta)\Delta_3\leq\deg(x)\leq(1+\delta)\Delta_3\} \bigr| \geq (1-\delta)|V(\H_i)|\quad a.a.s.
\]
\end{lemma}

\begin{proof}
Suppose $\lambda^2 = \frac{\sqrt{13\varepsilon}}{p^2(6p-\varepsilon)}$. In order to be able to use the Counting Lemma, we go back to the graph $G$. Suppose $uv$ is an edge in $G$. Similarly to what we did for quasirandom graphs, we say $uv$ is {\em balanced} if $\n(u,v)$ is at least $(1-\lambda)np^2$ and at most $(1+\lambda)np^2$; otherwise it is {\em unbalanced}. If more than $\lambda\e(G)$ edges are unbalanced, then we have
\[
   2\sum_{uv\in E(G)}|\n(u,v)-np^2| 
   \geq 2\lambda \e(G)\cdot \lambda np^2
   \geq \lambda^2np^2 (6p-\varepsilon)n^2
   = \sqrt{13\varepsilon}\, n^3.
\]
This contradicts Lemma \ref{lem:tri1}. We conclude that at least $(1-\lambda)\e(G)$ edges of $G$ are balanced.

Note that the graph $D_i$ has about $\e(G)/t$ edges. More precisely,
\[
(1-\varepsilon)\frac{\e(G)}{t}\leq |E(D_i)|\leq (1+\varepsilon)\frac{\e(G)}{t}, \quad \hbox{a.a.s.}
\]

By a similar argument as used before, we let $U$ be the number of balanced edges in $G$. Then $D_i$ contains at least $(1-\varepsilon)U/t$ balanced edges of $G$, and at least $(1-\varepsilon)^2U/t$ edges are contained in at least $(1-\varepsilon)(1-\lambda)np_1^2/4$ directed triangles and in at most $(1+\varepsilon)(1+\lambda)np_1^2/4$ triangles in $D_i$, a.a.s. Now we let $\delta\geq\max\{\lambda+\varepsilon+\varepsilon\lambda,\psi(\varepsilon,\lambda)\}$, where $\psi(\varepsilon,\lambda)$ is defined in Lemma \ref{lem:quasi1}. This yields the desired conclusion and completes the proof.
\end{proof}

Since the edges in $\H_i$ correspond to the directed triangles in $D_i$, Condition (2) in Theorem \ref{thm:hypergraphmatching} holds trivially. For Condition (3), we have the following lemma. The proof follows almost the same steps that were used in the proof of Lemma \ref{lem:c3}. We omit the details.

\begin{lemma}
Let $p,\varepsilon, t, \Delta_3$ and $\delta$ be as in Lemma \ref{lem:Delta_3 and delta}, and let $i\in [t]$. Suppose, moreover, that $\varepsilon < (p/4)^{12}$.
Let $P^\delta$ be the set of pairs of vertices $(u,v)$ in distinct sets $V_a,V_b$ for some $a,b\in \{1,2,3\}$, $a\ne b$, such that the number of directed paths from $u$ to $v$ of length $2$ is at least $(1+\delta)\Delta_3$.
Let $F_i$ be the number of directed triangles in $D_i$ which contain at least one directed edge $\overrightarrow{uv}\in P^\delta$. 
Then $F_i<\delta\Delta_3|V(\H_i)|$ a.a.s.
\end{lemma}

Now we are going to construct a near-triangular embedding of $G$ by using Lemma \ref{lem:rotation} and Theorem \ref{thm:hypergraphmatching}.

\begin{theorem}\label{thm:quasitriembed}
Let $0<p\le 1$ be a constant and $0<\varepsilon < (p/4)^{12}$. Then there is a constant $n_0$ such that for every natural number $n\ge n_0$, every graph $G\in\Q(n^{(3)},p,\varepsilon)$ admits a $9\varepsilon$-near-triangular embedding.
\end{theorem}

\begin{proof}
There exists a construction of $D_1,\dots,D_t$ such that the set $I\subseteq [t]$ containing all indices $i$ for which $\H_i$ fails to satisfy conditions (1)--(3) of Theorem \ref{thm:hypergraphmatching} is small, $|I|<\varepsilon t$. Let $J=[t]\setminus I$.

Now we apply Theorem \ref{thm:hypergraphmatching} for every $\H_i$ where $i\in J$. As before, there exists a matching $M_i$ of $\H_i$ of size at least $(1-\varepsilon)\frac{|E(D_i)|}{3}$, and another matching $M_i^\prime$ in $\H_i^{-1}$ which is disjoint with $M_i^{-1}$, and it also has size at least $(1-\varepsilon)\frac{|E(D_i)|}{3}$. Then in the digraph $D_i\cup D_i^{-1}$, $M_i\cup M_i^\prime$ has size at least $2(1-\varepsilon)\frac{|E(D_i)|}{3}$, and does not contain non-simple blossoms of length $2$.

As before, we are going to remove a cycle from each of the blossoms. For the short blossoms of length at most $1/\varepsilon$, we let $\overrightarrow{\mathcal{B}_j}$ be the $3$-blossom-graph of length $j$, and let $\mathcal{B}_j$ be its underlying simple graph. Then, by the Counting Lemma, 
\[
   \hom(\mathcal{B}_j,G)\leq 3n^{j+1}p^{2j}+2j\varepsilon n^{j+1}.
\]
For every $i\in [t]$, given a $3$-blossom simple graph $B_j$ contained in $G$, $\P(\overrightarrow{\mathcal{B}_j}\in D_i\cup D_i^{-1})=\frac{1}{2^{j-1}t^{2j}}$. Let $N_i(j)$ be the number of $\overrightarrow{\mathcal{B}_j}$ in $D_i\cup D_i^{-1}$. Similarly as before we have a.a.s. $N_i(j)\ll n^2p_1$, which implies $\sum_{j=1}^{1/\varepsilon}N_i(j) < \varepsilon(1-\varepsilon) 3n^2p_1 \le \varepsilon |E(D_i)|$, a.a.s.

Note that the number of blossoms of length at least $1/\varepsilon$ is bounded by $2\varepsilon|E(D_i)|$. (See the proof of Theorem \ref{thm:quasiembed} for details.)
Therefore, we can obtain a blossom-free subset of $M_i\cup M_i^\prime$ of size at least $2(1-6\varepsilon)\frac{|E(D_i)|}{3}$ a.a.s. 

Let $J'\subseteq J$ be the set of indices such that $|E(D_i)|\ge (1-\varepsilon)\e(G)/t$.
By Lemma \ref{lem:rotation} there exists a construction of $D_1,\dots,D_t$ such that $|J'|\ge (1-\varepsilon)|J| \ge (1-\varepsilon)^2 t$. Then $\sum_{i\in J'} |E(D_i)| \ge (1-\varepsilon)^3\e(G)$.
By using the blossom-free sets for all $D_i$, $i\in J'$, we obtain an embedding $\Pi$ of $G$ which has
\[
3f_3(\Pi)\geq 2(1-6\varepsilon)(1-\varepsilon)^3\e(G)\geq 2(1-9\varepsilon)\e(G).
\]
This completes the proof.
\end{proof}

We have the following consequence on the genus of quasirandom graphs in $\Q(n^{(3)},p,\varepsilon)$.

\begin{theorem}\label{thm:quasitrigenus}
Let $0<p\le 1$ be a constant and $0<\varepsilon < (p/4)^{12}$. There is a constant $n_0$ such that for every $n\ge n_0$, the genus of each graph $G\in\Q(n^{(3)},p,\varepsilon/18)$ satisfies:
\[
(1-\varepsilon)\frac{\e(G)}{6}\leq \g(G)\leq(1+\varepsilon)\frac{\e(G)}{6}
\]
and
\[
(1-\varepsilon)\frac{\e(G)}{3}\leq \ng(G)\leq(1+\varepsilon)\frac{\e(G)}{3}.
\]
\end{theorem}

\begin{proof}
By using the same argument as in Theorem \ref{thm:quasigenus}, the upper bound follows from Theorem \ref{thm:quasitriembed}, and the lower bound follows from \cite[Proposition 4.4.4]{top}.
\end{proof}

% ============================================================================

\subsection{Multipartite quasirandom graphs}
\label{subsect:multipartite QR}

We first consider the question how to partition a quasirandom graph into several quasi\-random subgraphs. We need this process throughout our approximation algorithm when we try to construct an embedding. Given a quasirandom graph $G$, the following lemma shows that we can partition the graph into a number of edge-disjoint graphs with prescribed edge densities, each of which is also quasirandom.

\begin{lemma}\label{lem:quasipartition}
Let $0<\varepsilon<1$, $0<p\le1$, let $k$ be a positive integer and $c_1,\dots,c_k$ be positive real numbers such that\/ $\sum_{i=1}^kc_i=1$. 
Then there exists $n_0$ such that for every $n\ge n_0$ and every $G\in\Q(n^{(2)},p,\varepsilon)$,
there exists an edge-partition of $G$ into $k$ edge-disjoint graphs $G_1,\dots,G_k$, such that for every $i\in[k]$, $G_i\in\Q(n^{(2)},c_ip,2c_i\varepsilon)$.
\end{lemma}

\begin{proof}
We first consider a random partition, that is, we put each edge in $G_i$ with probability $c_i$ independently at random. Let $V(G) = V_1\sqcup V_2$ be the bipartition of $G$ with $|V_1|=|V_2|=n$. Let us now consider some $G_i$, $i\in [k]$. Our goal is to prove that $\P(\d(G_i,c_iG) > c_i\varepsilon) < 1/k$. If the cut distance of $G_i$ from $K(n^{(2)}, c_ip)$ is larger than $c_i\varepsilon$, there are subsets $X\subseteq V_1$ and $Y\subseteq V_2$ whose discrepancy in the definition of the cut distance (\ref{cut2}) is too large:
\begin{equation}
    \Bigl|\e_G(X,Y) - c_ip\,|X|\, |Y|\Bigr| > c_i\varepsilon \, n^2. \label{eq:0*}
\end{equation}
Any such subsets $X,Y$ would need to be large,
$|X||Y| \ge c_i\varepsilon n^2$, as explained after the definition (\ref{cut2}).
Chebyshev's inequality implies that in a random partition, the following holds a.a.s.:
\begin{equation}
(1-\varepsilon)c_i\, \e_G(X,Y)\leq \e_{G_i}(X,Y) \leq (1+\varepsilon)c_i\, \e_G(X,Y). \label{eq:1*}
\end{equation}
That is,
\begin{equation}
|\e_{G_i}(X,Y)-\e_G(X,Y)c_i|\leq c_i\varepsilon \e_G(X,Y) \le c_i\varepsilon n^2. \label{eq:2*}
\end{equation}
This implies that $\d(G_i,c_iG))\le c_i\varepsilon$ a.a.s. Thus, if $n$ is large enough, $\P(\d(G_i,c_iG) > c_i\varepsilon) < 1/k$. As a consequence we derive that there is a partition such that for each $i\in[k]$, $\d(G_i,c_iG) \le c_i\varepsilon$.

Finally, 
\begin{eqnarray*}
  \d(G_i,K(n^{(2)},c_ip)) &\le& \d(G_i,c_iG)) + \d(c_iG,K(n^{(2)},c_ip)) \\
  &=& \d(G_i,c_iG)) + c_i\, \d(G,K(n^{(2)},p)) \\
  &\le& 2c_i\varepsilon.
\end{eqnarray*}
Thus, $G_i\in\Q(n^{(2)},c_ip,2c_i\varepsilon)$ for every $i\in[k]$.
\end{proof}

Let us observe that the bound $n_0$ in Lemma \ref{lem:quasipartition} is ``monotone" in the sense that for any positive $a,b,c$, we can take the same $n_0 = n_0(a,b,c)$ whenever $\varepsilon\ge a$, $p\ge b$, and $c_i\ge c$ ($i\in[k]$).  

For an edge-weighted simple graph $H$ of order $m$, let $w_e\ge0$ be the weight of its edge $e$. Let $\T$ be the set of all triangles in $H$ (of positive edge-weight). Now we consider the following linear program with optimum solution $\{t(T)\mid T\in\T\}$ and its maximum $\nu^*(H)$:
\begin{equation}\label{lp}
\begin{split}
&\nu^*(H)=\max\sum_{T\in\T}t(T),\\
&\sum_{T\ni e,T\in\T}t(T)\leq w_e, \quad \text{ for every edge }e\text{ of }H,\\[1mm]
&t(T)\geq0,\quad \text{ for every }T\in\T.
\end{split}
\end{equation}

Given $G\in\Q(n^{(M)},P,\varepsilon)$, let $p_{ij}$ be the $(i,j)$-entry of $P$. As usual, we assume that $V(G) = V_1\sqcup\dots\sqcup V_M = V(K((n^{(M)},P))$, where $|V_i| = n$ for $1\le i\le M$ so that $\d(G,K((n^{(M)},P)) \le \varepsilon$ as defined by (\ref{cut1}). Let $H$ be the quotient graph of $G$. That is, $V(H) = [M]$, and for every $i,j\in V(H)$, its edge-weight is $w_{ij} = p_{ij}$. 
The optimum value $\nu^*(H)$ of the linear program (\ref{lp}) gives an upper bound on the packing of triangles in $G$.

\begin{lemma}
\label{lem:packing triangles}
Let $G\in\Q(n^{(M)},P,\varepsilon)$. Suppose that $\F$ is a set of triangles in $G$ such that each edge lies in at most $\lambda$ of them. Then 
\begin{equation}\label{eq:f}
    |\F| \leq (1+M^2\varepsilon)\lambda\, \nu^*(H) n^2.
\end{equation}
\end{lemma}

\begin{proof}
For any distinct $i,j,k\in [M]$, let $f(i,j,k)=\bigl|\F\cap G[V_i,V_j,V_k]\bigr|$, that is, the number of triangles in $\F$ which are inside the subgraph $G[V_i,V_j,V_k]$. Let $f(i,j)$ be the number of triangles in $\F$ containing an edge in $G[V_i,V_j]$. Then we have
\begin{equation}\label{eq:4*}
f(i,j)=\sum_{k\neq i,j}f(i,j,k).
\end{equation}
By the assumption that each edge lies in at most $\lambda$ triangles in $\F$ and by Lemma \ref{lem:QR B_ij and T_ijk}, we have the following inequality:
\begin{equation}\label{eq:5*}
f(i,j)\leq \lambda \e_G(V_i,V_j) \le \lambda (1+M^2\varepsilon) n^2 p_{ij}.
\end{equation}

For a triangle $T_{ijk}\in \T$ corresponding to $V_i,V_j,V_k$, let 
$t(T_{ijk}) = f(i,j,k) / (\lambda(1+M^2\varepsilon)n^2)$. By (\ref{eq:4*})--(\ref{eq:5*}), we see that for $i,j\in [M]$, $i\ne j$, 
$$
  \sum_{k\in [M]\setminus\{i,j\}} t(T_{ijk}) = \frac{f(i,j)}{\lambda(1+M^2\varepsilon)n^2} \le p_{ij}.
$$
Hence, $\{t(T_{ijk})\}$ is a feasible solution of the linear program (\ref{lp}).
Then (\ref{eq:f}) follows by the definition of $\nu^*(H)$.
\end{proof}

Lemma \ref{lem:packing triangles} will be used in the proof of the next theorem with $\F$ being the set of facial triangles of an embedding of the graph $G$ (with $\lambda=2$) in order to obtain a lower bound on the genus of $G$.

With all tools in hand, we have the following theorem on the genus of $W_M$-quasirandom graphs.

\begin{theorem}\label{thm:quasimultigenus}
For every $\varepsilon_0 > 0$, integer $M_0>1$ and $c_0>0$ there exists a positive integer $n_0$ such that the following holds. Suppose that $\varepsilon_0\le \varepsilon \le \frac{1}{2}$, $M\le M_0$ and $n\ge n_0$ and let $c_1 = \varepsilon_0 c_0/M^3$ and $\varepsilon' = \min\{ \tfrac{1}{60}\varepsilon M^{-2}, (\tfrac{c_1}{4})^{16}/10 \}$.  Let $G$ be a graph in $\Q(n^{(M)},P, \varepsilon')$, where each nonzero entry $p_{ij}$ of $P$ is larger than $c_0$. Suppose $H$ is the quotient graph of $K(n^{(M)},P)$ and let $\nu^*(H)$ be the optimum of the linear program (\ref{eq:lp}). Then we have
\[
  (1-\varepsilon)\frac{\e(G)-\nu^*(H)n^2}{4}\leq\g(G) < 
  (1+\varepsilon)\frac{\e(G)-\nu^*(H)n^2}{4} + \tfrac{1}{2}nM^3
\]
and
\[
  (1-\varepsilon)\frac{\e(G)-\nu^*(H)n^2}{2}\leq\ng(G) <
  (1+\varepsilon)\frac{\e(G)-\nu^*(H)n^2}{2}+nM^3.
\]
\end{theorem}

\begin{proof}
We first consider the upper bound. Suppose $G$ is defined on the vertex-set $V_1\sqcup \dots\sqcup V_M$. Let us consider the triangles $\{T_{ijk}\}$ for every $1\leq i<j<k\leq M$, and let $d_{ijk} = t(T_{ijk})$ be an optimal solution of (\ref{eq:lp}). For any distinct $i,j\in [M]$, let $b_{ij} = p_{ij}-\sum_{k\neq i,j}d_{ijk}$. (Here and below, we consider each multi-index $ijk$ or $ij$ as an unordered triple, so we do not have to assume that $i<j<k$.) For every $1\leq i<j\leq M$ with $p_{ij} > 0$, we will partition the edges between $V_i$ and $V_j$ into at most $M-1$ parts $E_1^{ij},\dots,E_M^{ij},E_{0}^{ij}$ (excluding $E_i^{ij}$ and $E_j^{ij}$), using Lemma \ref{lem:quasipartition} with $d_{ij1}/p_{ij},\dots,d_{ijM}/p_{ij}$, and $b_{ij}/p_{ij}$ as the constants $c_1,\dots,c_k$ in the lemma.

In the rest of the proof we will apply several results from the first part of the paper. In those results, it is usually requested that the density $p$ is large when compared to $\varepsilon$. But the densities $d_{ij1}/p_{ij},\dots,d_{ijM}/p_{ij}$, and $b_{ij}/p_{ij}$, which we will use, may be too small. When they are really small, we can neglect them by using the $\varepsilon$-term in the upper bound. Those that cannot be neglected may still be too small and we will have to use a different $\varepsilon$. All this will have effect on the requirement how large $n_0$ should be. The important request for us is that each time, $n_0$ should only depend on $\varepsilon_0, M_0$ and $c_0$ and that the number of times these estimates are applied, also depends on $\varepsilon_0, M_0$ and $c_0$ only.

Let us observe that $\e(G) = \sum_{i<j}p_{ij}n^2(1\pm o(1))$ and that $\nu^*(H) \le \tfrac{1}{3} \sum_{i<j}p_{ij}$. Thus,
$$
  \e(G) - \nu^*(H)n^2 \ge \tfrac{2}{3} \sum_{i<j}p_{ij}n^2(1-o(1)) \ge \tfrac{1}{2}\e(G)
$$
for sufficiently large $n$. If some densities $d_{ijk}$ or $b_{ij}$ defined above are smaller than $c_1$, then we will neglect the corresponding subgraph $G_{ijk}$ or $G_{ij}$ defined below (respectively). We observe that all such neglected subgraphs have less than $\tfrac{\varepsilon}{8}(\e(G) - \nu^*(H)n^2)$ edges.

By Lemma \ref{lem:quasipartition} (and the remark after the lemma about the ``monotonicity" of the bounds), the edges of $G$ can be partitioned in such a way that we obtain $O(M^2)$ $\varepsilon/20$-quasirandom bipartite graphs and $O(M^3)$ $\varepsilon/20$-quasirandom tripartite graphs if $n$ is large enough.  That is, pick $V_i$, $V_j$ and $V_k$, there exist a subgraph $G^{ij}_k$ defined on $V_i\cup V_j$, a subgraph $G^{jk}_i$ defined on $V_j\cup V_k$ and a subgraph $G_j^{ik}$ defined on $V_i\cup V_k$ such that $G_k^{ij},G_i^{jk},G_j^{ik}\in\Q(n^{(2)},d_{ijk},\varepsilon/20)$. Combining them together, we obtain a tripartite graph $G_{ijk} \in \Q(n^{(3)},d_{ijk},\varepsilon/20)$. Similarly, each $E_{0}^{ij}$ defines a bipartite graph $G_{ij} \in \Q(n^{(2)},b_{ij},\varepsilon/20)$. 
The only condition here is that $n$ is large enough, and this is satisfied by taking $n_0$ that is sufficiently large. 
Note that the requested value of $n_0$ comes from Lemma \ref{lem:quasipartition} and from Theorems~\ref{thm:hypergraphmatching}, \ref{thm:quasibigenus}, and \ref{thm:quasitrigenus}, thus it only depends on $\varepsilon_0, M_0, c_0$, when we make use of these results only for those subgraphs $G_{ij}$ and $G_{ijk}$, whose density is at least $c_1$.

Now we embed the graph $G$ by using the partition we constructed. For the tripartite parts, we embed them as quasirandom tripartite graphs, and for the bipartite parts, we embed them as quasirandom bipartite graphs, using Theorems \ref{thm:quasitrigenus} and \ref{thm:quasibigenus} (respectively). We obtain a rotation system $\Pi$ for the disjoint union of these graphs whose genus is
\[
\begin{split}
\sum_{ij}\g(G_{ij}) + \sum_{ijk}\g(G_{ijk})
&\leq (1+\varepsilon/3)\Big(\sum_{i<j}\frac{n^2b_{ij}}{4}+\sum_{i<j<k}\frac{3n^2d_{ijk}}{6}\Big)\\
&= (1+\varepsilon/3)\Big(\sum_{i<j}\frac{n^2p_{ij}}{4}-\sum_{i<j<k}\frac{n^2d_{ijk}}{2}\Big)\\
&\le (1+\varepsilon/2)\frac{\e(G)-\nu^*(H)n^2}{4}.
\end{split}
\]

To obtain an embedding of $G$ we identify copies of the same vertex in different copies of the quasirandom subgraphs. Each identification can increase the genus by at most $1$, and the number of identifications is at most $\tfrac{1}{2}M^2$ for each of the vertices. This is then accounted for in the added term $\tfrac{1}{2}nM^3$ in the upper bound. Finally, we add all edges from the neglected subgraphs $G_{ijk}$ and $G_{ij}$, whose densities were too small. This increases the genus for at most $\tfrac{\varepsilon}{8}(\e(G) - \nu^*(H)n^2)$, yielding the desired upper bound.

It remains to justify the lower bound for $\g(G)$. 
Suppose $\Pi$ is the minimum genus embedding of $G$, and let $f_3(\Pi)$ be the number of triangular faces of $\Pi$, and let $f'(\Pi)$ be the number of non-triangular faces of $\Pi$. 
By Lemma \ref{lem:packing triangles}, $f_3(\Pi) \le 2(1+\tfrac{\varepsilon}{60})\nu^*(H)n^2$. 
Since $2\e(G)\geq3f_3(\Pi)+4f'(\Pi)$, we have
\[
f'(\Pi)\leq\frac{2\e(G)-3f_3(\Pi)}{4}.
\]
Euler's formula now shows that
\begin{align*}
\g(G)&=\frac{\e(G)-n-f_3(\Pi)-f'(\Pi)+2}{2}\\
&\geq\frac{2\e(G)-f_3(\Pi)-4n+8}{8}\\
&\geq\frac{\e(G)-(1+\tfrac{1}{60}\varepsilon)\nu^*(H)n^2-2n+4}{4}.
\end{align*}
To complete the proof, we only need to show that the last line in the above series of inequalities is at least as large as $\tfrac{1}{4}(1-\varepsilon)(\e(G)-\nu^*(H)n^2)$. To see this, it suffices to prove that 
\begin{equation}\label{eq:last line}
   \tfrac{1}{60}\varepsilon\nu^*(H)n^2 + 2n \le \varepsilon(\e(G) - \nu^*(H)n^2).
\end{equation}

We consider two cases. First, if $\nu^*(H)=0$, then we just use the fact that $P$ has at least one non-zero entry, which is larger or equal to $c_0$. Therefore, $\e(G)\ge (1-\varepsilon)c_0n^2\ge \frac{1}{2}c_0n^2$. If $n_0\ge 4(\varepsilon_0 c_0)^{-1}$, this implies (\ref{eq:last line}).

Suppose now that $\nu^*(H)\ne 0$. Then $\nu^*(H)\ge c_0$ and as above we see that $\e(G)\ge 3\nu^*(H)n^2(1-\varepsilon)$. Again, if $n\ge n_0\ge (\varepsilon_0 c_0)^{-1}$, this implies (\ref{eq:last line}).
\end{proof}

%\textcolor{red}{
%We remark that the current statement of Theorem~\ref{thm:quasimultigenus} only holds when $G$ is a quasirandom $m$-partite graph when $m$ is bounded above by a constant. The reason is, we require the graph between any pair of vertex sets be random-like, and also after we partition the graph into several small parts using the optimal solution of (\ref{eq:lp}), we want each part to be reasonably dense. The first issue can be fixed by finding an $\varepsilon$-regular partition of the graph; and the second issue can be fixed using the result proved in the next section. }

Let us observe that Theorem \ref{thm:quasimultigenus} gives a formula for the genus of $W$-random graphs, where $W$ is any nontrivial graphon (see \cite{LoSos}), since $W$ can be approximated by an $m$-partite quasirandom graph. This aspect may be of independent interest in topological graph theory and the theory of graph limits \cite{graphlimit}.

%%%%%%%%%%%%%%%%%%%%%%%%%%%%%%%%%%%%%%%%%%%

\section{Bounded fractional packing}

Given a graph or a multigraph $G$ and a small graph $F$, the \emph{$F$-packing number}, denoted by $\nu_F(G)$, is the maximum number of pairwise edge-disjoint copies of $F$ in $G$.  By a \emph{copy of $F$ in $G$} we mean a set of edges\footnote{With this definition we implicitly assume that $F$ contains no isolated vertices. If that were not the case, we would add isolated vertices to $H$, and the proof of Theorem \ref{thm: packing weighted graph} would go through.} $H\subseteq E(G)$ with $|H|=|E(F)|$ that together form a subgraph of $G$ isomorphic to $F$. The set of all copies of $F$ in $G$ will be denoted by $\binom{G}{F}$.  A \emph{fractional $F$-packing} is a function $\psi$ assigning to each copy of $F$ in $G$ a nonnegative number such that for every edge $e\in E(G)$ we have $\sum \{\psi(H)\mid H\in \binom{G}{F}, e\in H\} \leq 1$.  This notion extends to edge-weighted graphs, where the condition is replaced by 
\begin{equation}
    \sum \Big\{\psi(H)\mid H\in \binom{G}{F}, e\in H\Big\} \leq w(e) \label{eq:fractional packing constraint}
\end{equation}
with $w(e)$ being the weight of the edge.  The \emph{fractional $F$-packing number}, denoted by $\nu^*_F(G)$, is then defined to be the maximum value of $\sum_{H\in \binom{G}{F}}\psi(H)$. 

Haxell and R\"odl~\cite{HaxellRodl01} proved that $\nu^{*}_F(G)-\nu_F(G)=o(|V(G)|^2)$ for any graph $G$. Our next result extends this theorem to multigraphs with bounded edge-multiplicities. The proof follows the argument by Yuster~\cite{Yuster05} in his simplified proof of the Haxell--R\"odl theorem.

\begin{theorem}\label{thm: packing weighted graph}
Let $F$ be a graph on $k$ vertices. Let $G$ be a multigraph of order $n$ with each edge-multiplicity at most $s$.
Then $\nu_F^*(G)-\nu_{F}(G) = o_{s,k}(n^2).$
\end{theorem}

\begin{proof}
Let $\psi$ be a fractional $F$-packing in $G$ that achieves the minimum value $\nu^*_F(G)$. Let $G^*$ be the edge-weighted simple graph with $V(G^*) = V(G) =: V$, and for every pair of vertices $u,v$, the edge-weight $w$ of $uv$ is equal to the number of edges between $u$ and $v$ in $G$.  In particular, we have $w:\binom{V}{2}\to [0,s]$.  Take $\varepsilon>0$ and apply the Regularity Lemma\footnote{Here we need the Regularity Lemma for weighted graphs \cite{graphlimit} which works in exactly the same way as for simple graphs when the edge-weights are bounded.} to $G^*$ with respect to $\varepsilon$. As we may assume that $n>1/\varepsilon$, we may also require that the obtained $\varepsilon$-regular partition $\mathcal{P} = \{V_1,\dots,V_K\}$ of $G^*$ has $K>1/\varepsilon$ parts. Let $G'$ be the weighted graph obtained from $G^*$ by removing edges from each $G^*[V_i]$, edges between irregular pairs, and edges between pairs of density smaller than $\varepsilon$. Let $\psi'$ be the restriction of the fractional packing $\psi$ to copies of $F$ in $G'$. Each removed edge $e$ decreases the total sum of $\psi'(H)$ at most by $w(e)\le s$. We view $\psi'$ also as a fractional packing in $G^*$ be setting $\psi'(H)=0$ if $H\in \binom{G^*}{F}\setminus \binom{G'}{F}$. It is easy to see that
\begin{equation}\label{eq: frac packing}
  \nu_F^*(G) - \sum_{H\in\binom{G^*}{F}}\psi'(H) \leq s\bigl( K(\tfrac{n}{K})^2 + \varepsilon K^2 (\tfrac{n}{K})^2 + \varepsilon K^2 (\tfrac{n}{K})^2\bigr) \leq 3\varepsilon s n^2.
\end{equation}

Now we randomly partition each $V_i$ into $t$ parts of (almost) equal size, where $t = \lceil k^2/\varepsilon^{1/2}\rceil$. We obtain a refined vertex partition $\mathcal{Q}=\{U_1,\dots,U_M\}$, where $M=tK$. Note that, for every $U_p\subseteq V_i$ and $U_q\subseteq V_j$ (with $p\ne q$) and for $X\subseteq U_p$ and $Y\subseteq U_q$ with $|X|>t\varepsilon |U_p|$ and $|Y|>t\varepsilon |U_q|$, we have $|d(X,Y)-d(V_i,V_j)|<\varepsilon$ since $(V_i,V_j)$ is $\varepsilon$-regular if $i\ne j$, and there are no edges if $i=j$. Since the edge-weights are bounded by the constant $s$, Chebyshev's inequality implies that $|d(V_i,V_j)-d(U_p,U_q)|<\varepsilon$ a.a.s., and hence $|d(X,Y)-d(U_p,U_q)| < 2\varepsilon < t\varepsilon$.  This implies that with high probability, $(U_p,U_q)$ is an $\varepsilon'$-regular pair, where $\varepsilon' = t\varepsilon$. Thus, a.a.s., $\mathcal{Q}$ is an $\varepsilon'$-regular partition of $G'$, in which each pair is $\varepsilon'$-regular.

We say that a copy of $F$ in $G'$ is \emph{uniform} if no two of its vertices lie in the same part of $\mathcal{Q}$. Let $\psi''$ be the restriction of $\psi'$ to all uniform copies of $F$, i.e. we set $\psi''(H)=0$ if $H$ is not uniform. Given $H$, a copy of $F$ in $G'$, let $N(H)$ be the number of pairs of vertices of $H$ that lie in the same part of $\mathcal{Q}$. Clearly,
\[
\E(N(H)) \leq \binom{k}{2}/t \leq \varepsilon^{1/2}.
\]
By Markov's inequality, the probability that $H$ is not uniform is at most $\varepsilon^{1/2}$. By linearity of expectation, we therefore have
\begin{align*}
\E \Big(\sum_{H\in\binom{G}{F}}\psi''(H)\Big) &\geq (1-\varepsilon^{1/2}) \sum_{H\in\binom{G}{F}}\psi'(H) \\
  & \ge \sum_{H\in\binom{G}{F}}\psi'(H) - \varepsilon^{1/2} \tfrac{1}{3}\e(G) \\
  & \ge \nu_F^*(G) - 3\varepsilon sn^2 - \tfrac{1}{6}\varepsilon^{1/2} sn^2.
\end{align*}
For the second inequality we used (\ref{eq:fractional packing constraint}) and for the third one we used (\ref{eq: frac packing}).
This implies that there exists an equitable partition $\mathcal{Q}$ that is $\varepsilon'$-regular (in which all pairs $(U_i,U_j)$ are $\varepsilon'$-regular) and that the fractional packing $\psi''$ that corresponds to $Q$ satisfies the inequality:
$$
   \sum_{H\in\binom{G}{F}}\psi''(H) \ge \nu_F^*(G) - 3\varepsilon sn^2 - \varepsilon^{1/2} sn^2.
$$
We now fix such a regular partition $\mathcal{Q}$.

Let $R$ denote the quotient graph $G'/\mathcal{Q}$, that is, $V(R)=\{1,\dots,M\}$, and $ij\in E(R)$ if $(U_i,U_j)$ is $\varepsilon'$-regular in $G'$ and $d_{G'}(U_i,U_j) > 0$. We define a fractional $F$-packing $\psi_R$ on $R$, such that for every copy $F_0$ of $F$ in $R$, if $F_0 = \{i_1j_1,\dots,i_r j_r\}$ ($r=|F_0|=|E(F)|$), define $\psi_R(F_0)$ be the total $\psi''$-weight of (uniform) copies of $F$ in $G'[U_{i_1},U_{j_1}]\cup \cdots \cup G'[U_{i_r},U_{j_r}]$ divided by $n^2/M^2$. Let $\mathcal{F}_{ij}$ be the collection of copies of $F$ in $R$ which contain the edge $ij$.

Next, we consider the original multigraph $G$. For every edge $e$ between $U_i,U_j$ in $G$, we color the edge $e$ with $F_0\in \mathcal{F}_{ij}$ that is chosen at random, with probability $\psi_R(F_0)/d_{ij}$, where $d_{ij}$ is the edge density between $U_i$ and $U_j$ in $G'$. Note that some of the edges of $G$ may stay uncolored; this may be because the edge is not used in $G'$, or because $\mathcal{F}_{ij}$ could be empty.

Now we fix a copy $F_0$ of $F$ in $R$ with $\psi_R(F_0) > 1/M^{k-1}$, and assume its vertex-set is $\{1,\dots,k\}$. Let $L = G'[U_{i_1},U_{j_1}]\cup \cdots \cup G'[U_{i_r},U_{j_r}] \subseteq G'$.
By a standard application of the counting lemma (see, e.g.\ Lemma 15 in \cite{HaxellRodl01}), there is a spanning subgraph $L'$ of $L$, such that for every edge $e\in E(L')$, if $e\in E(U_i,U_j)$, then the number of uniform copies of $F$ in $L$ that contain $e$ is in the interval
\[
\bigg[\Big(\frac{n}{M}\Big)^{k-2} \frac{\prod_{ab\in F_0}d_{ab}}{d_{ij}} - \varepsilon' \Big(\frac{n}{M}\Big)^{k-2},\Big(\frac{n}{M}\Big)^{k-2}\frac{\prod_{ab\in F_0}d_{ab}}{d_{ij}} + \varepsilon' \Big(\frac{n}{M}\Big)^{k-2}\bigg].
\]

Let $L_{F_0}$ be the spanning subgraph of $L'$ containing edges which are colored by $F_0$. Then $L_{F_0}$ is a random subgraph of $L'$. For every edge $e\in E(L_{F_0})$, let $\xi_{F_0}(e)$ be the number of copies of $F$ in $L_{F_0}$ that contain $e$. Using Chernoff's bound, we conclude that a.a.s.
\[
\bigg| \xi_{F_0}(e) - \Big(\frac{n}{M}\Big)^{k-2}\psi_R(F_0)^{r-1} \bigg| < \varepsilon' \Big(\frac{n}{M}\Big)^{k-2} \psi_R(F_0)^{r-1}.
\]
The proof of the above inequality follows the proof of Lemma 3.1 in \cite{Yuster05}, and we omit the details here.

On the other hand, for any edge $ij\in F_0$ in $R$, the expected number of edges between $U_i$ and $U_j$ that belongs to $L_{F_0}$ is
\[
 d_{ij}\frac{\psi_R(F_0)}{d_{ij}}\frac{n^2}{M^2} = \psi_R(F_0)\frac{n^2}{M^2}.
\]
Thus by summing over all edges of $F_0$ and by Chernoff's bound, we also conclude that 
$|E(L_{F_0})| \geq (1-2\varepsilon') r\psi_R(F_0)\frac{n^2}{M^2}$ a.a.s.
We now fix a coloring that satisfies all the above inequalities.

Let us consider an $r$-uniform hypergraph $\mathcal{H}_{F_0}$ whose vertices are the edges in $L_{F_0}$, and the hyperedges of $\mathcal{H}_{F_0}$ are the uniform copies of $F$ in $L_{F_0}$. Note that $\mathcal{H}_{F_0}$ satisfies the conditions of Theorem~\ref{thm:hypergraphmatching}. Therefore, $\mathcal{H}_{F_0}$ has an $\varepsilon'$-near perfect matching. In particular, the matching has size at least
\[
(1-\varepsilon')\frac{|E(\mathcal{H}_{F_0})|}{r} \geq (1-3\varepsilon')\frac{\psi_R(F_0)n^2}{M^2}.
\]
By taking all such $F_0$ in $R$, we will get an $F$-packing $\phi$ in $G$.
For copies of $F$ in $R$ that have weight less than $1/M^{k-1}$, the total contribution of such copies of $F$ is at most
\[
M^k\frac{1}{M^{k-1}}\Big(\frac{n}{M}\Big)^2 < \varepsilon n^2.
\]
Thus by \eqref{eq: frac packing}, we have $\nu_F^*(G)-\nu_F(G) \leq \nu_F^*(G)-\sum_{H\in\binom{G}{F}}\phi(H) < 100sk^2\sqrt\varepsilon n^2$. 
The theorem follows by taking $\varepsilon\to 0$.
\end{proof}

\section{Algorithms and analysis}%%%
\label{sect:Algorithms}

In this section, we will show how to use Theorem \ref{thm:quasimultigenus} to find a near minimal embedding of a large dense graph from the algorithmic point of view. Section \ref{sect:Algorithms-genus} provides a deterministic EPTAS for the problem {\sc Approximating Genus Dense}. In the rest of the paper (\ref{sect:Algorithms-embedding}), we will discuss how to construct a near minimum genus embedding for the problem {\sc Approximate Genus Embedding Dense}.

\subsection{Genus of dense graphs}
\label{sect:Algorithms-genus}

Suppose $G$ has $n$ vertices and suppose we have an equitable partition $\mathcal{P}$ of $V(G)$ into $K$ parts, $V(G)=V_1\sqcup\dots\sqcup V_K$. We use $H=G/\mathcal{P}$ to denote the corresponding edge-weighted quotient graph with $K$ vertices, $V(H)=\{v_1,\dots,v_K\}$. The edge between $v_i$ and $v_j$ in $H$ has weight equal to the edge density $d_{ij}=d(V_i,V_j)$ between $V_i$ and $V_j$ in $G$.

\begin{theorem}
\label{thm:genus}
For every $\varepsilon>0$, there exist a positive number $\varepsilon^\prime>0$ and a natural number $n_0$ such that for every graph $G$ with $n=|V(G)|\ge n_0$ vertices and at least $\varepsilon n^2$ edges, there is 
an $\varepsilon'$-regular partition $\mathcal{P}$ with $K$ parts as defined in Theorem \ref{thm:regularity}, such that the following holds. Let $H=G/\mathcal{P}$ be the quotient graph, from which we remove all edges between irregular pairs, and all edges whose weight is less than $\varepsilon/10$, and let $\nu^*(H)$ be the optimal value of the linear program (\ref{lp}). Then
\[
(1-\varepsilon)\frac{\e(G)-\nu^*(H)n_1^2}{4} \leq \g(G) \leq (1+\varepsilon)\frac{\e(G)-\nu^*(H)n_1^2}{4},
\]
and
\[
(1-\varepsilon)\frac{\e(G)-\nu^*(H)n_1^2}{2} \leq \ng(G) \leq (1+\varepsilon)\frac{\e(G)-\nu^*(H)n_1^2}{2},
\]
where $n_1=n/K$.
\end{theorem}

\begin{proof}
Let $\varepsilon_1 = \varepsilon/10$ and $\varepsilon'< (\varepsilon_1/4)^{17}$. Apply the  regularity lemma with respect to $\varepsilon^\prime$ to $G$.
Let $K>1/\varepsilon'$ be the number of parts in the resulting $\varepsilon'$-regular partition $\mathcal{P}=V_1\sqcup\dots\sqcup V_K$.
We remove all the edges in each $G[V_i]$ for every $i\in[K]$, the edges between irregular pairs, and the edges between pairs of edge density smaller than $\varepsilon_1$. The number of edges that were removed is at most
\[
K\cdot \frac{1}{2}\Big(\frac{n}{K}\Big)^2+\varepsilon'K^2\Big(\frac{n}{K}\Big)^2+\varepsilon_1 K^2\Big(\frac{n}{K}\Big)^2\leq 2\varepsilon_1 n^2.
\]
Let $G'$ be the resulting graph. 
Since the genus is edge-Lipschitz (adding or deleting an edge changes the genus by at most 1), we have that
\begin{equation}
    0\leq \g(G)-\g(G')\leq 2\varepsilon_1 n^2.
    \label{eq:star6}
\end{equation}

For the edge-weighted quotient graph $H$ of order $K$, the linear program (\ref{lp}) yields a real number $\nu^*(H)$. Also, for each pair $(V_i,V_j)$ in $\mathcal{P}$ of $V(G)$, let $d_{ij}$ be the edge density between $V_i$ and $V_j$ in $G$, and let $m_{ij}$ be the largest integer such that $m_{ij}\varepsilon_1\leq d_{ij}$. Let $H^*$ be a multigraph such that $V(H^*)=V(H)$, and the edge multiplicity between vertices $i$ and $j$ is $m_{ij}$. Note that the edge multiplicity of $H^*$ is upper bounded by $m=1/\varepsilon_1$. By Theorem~\ref{thm: packing weighted graph} we have
\[
\nu^*_{K_3}(H^*)-\nu_{K_3}(H^*)<\varepsilon' K^2
\]
if $n$ is sufficiently large, say $n\ge n_1$. The bound $n_1$ depends on $\varepsilon'$ and $s$, so it only depends on $\varepsilon$ after all.

By the construction of $H^*$, we also get $\nu(H)-\nu^*_{K_3}(H^*)/m<\varepsilon_1 K^2$. Let $\Psi: \binom{H^*}{K_3}\to \{0,1\}$ be the packing corresponding to $\nu_{K_3}(H^*)$. Given $i, j\in V(H^*)$, let $\mathcal{T}_{ij}$ be the collection of triangles in the support set of $\Psi$ that contains $ij$, and let $t_{ij}=|\mathcal{T}_{ij}|$. Following the proof of Theorem \ref{thm:quasimultigenus}, we partition the edge set $E(V_i,V_j)$ in $G'$ randomly into $t_{ij}+1$ parts, with probabilities $1/m_{ij},\dots,1/m_{ij},(m_{ij}-t_{ij})/m_{ij}$, respectively. Note that by blowing up $H^*$, this random edge-partition as well as $\Psi$ give us a decomposition of $G'$ into quasirandom bipartite graphs and quasirandom tripartite graphs. In order to apply the same proof as in Theorem \ref{thm:quasimultigenus}, it remains to check the quasirandom graph satisfies the condition required in Theorem~\ref{thm:quasibigenus} and Theorem~\ref{thm:quasitrigenus}. Since $(V_i,V_j)$ is $\varepsilon'$-regular, by Lemma~\ref{lem:quasipartition}, the bipartite graphs and tripartite graphs we obtained are $\varepsilon'/\varepsilon_1$-quasirandom.
By applying the same proof we used in proving Theorem~\ref{thm:quasimultigenus}, we conclude that
\[
  (1-\varepsilon')\frac{\e(G)-\nu^*(H)n_1^2}{4}\leq\g(G) < 
  (1+\varepsilon')\frac{\e(G)-\nu^*(H)n_1^2}{4} + \tfrac{1}{2}n_1K^3.
\]
Note that $K$ only depends on $\varepsilon'$, by the choice of $\varepsilon'$, the second term $\tfrac{1}{2}n_1K^3$ is linear on $n_1$, and it is absorbed by the dominating term by replacing $\varepsilon'$ to $\varepsilon$. Note that this yields another lower bound $n_0'$ on $n$. By taking $n_0 = \max\{n_1,n_0'\}$, the theorem follows.
\end{proof}

With all tools in hand, we obtain an algorithm for {\sc Approximating Genus Dense}, whose running time is quadratic. Given $\varepsilon>0$ and a graph $G$ of order $n$, we do the following:\medskip

{\sc Step 0.} Let $n_0=n_0(\varepsilon)$ be the constant from Theorem \ref{thm:genus}. If $n=|V(G)|<n_0$, then determine $g(G)$ by using some known algorithm for the genus of graphs, and then stop. Otherwise proceed with Step 1.\medskip

{\sc Step 1.} Let $\tau=3\varepsilon^\prime/2$, where $\varepsilon^\prime$ is defined in Theorem \ref{thm:genus}. Pick $\alpha=1/2$, $m=4/\varepsilon$ and apply the algorithm in Theorem \ref{thm:algpartition} with integer $k$ taking values from $m
M$ to $s(\tau,M)$. Then the algorithm will output an $\varepsilon^\prime$-regular partition into $K$ parts, where $M\leq K\leq s(\tau,M)$.\medskip

{\sc Step 2.} Consider the quotient graph $H=G/\mathcal{P}$. Solve the linear program (\ref{lp}) on $H$ to obtain $\nu^*(H)$.\medskip

{\sc Step 3.} Output $g=\tfrac{1}{4}(1+\varepsilon)(\e(G)-\nu^*(H)n_1^2)$, where $n_1=n/K$.

\begin{corollary}
The problem {\sc Approximating Genus Dense} can be solved in $O(f(\varepsilon) n^2)$ time, where $n$ is the order of the input graph and $f:\RR^+\to\RR^+$ is an explicit function.
\end{corollary}

\subsection{Embeddings of dense graphs}
\label{sect:Algorithms-embedding}

Now, we turn to our algorithm for {\sc Approximate Genus Embedding Dense} where the added feature is to construct an embedding. Given $\varepsilon>0$ and a graph $G$ of order $n$, we can apply {\sc Approximating Genus Dense} to get $g$ such that $\g(G)\leq g\leq(1+\varepsilon)\g(G)$. We are going to construct a rotation system $\Pi$ of $G$, whose genus satisfies the same bound. Our algorithm proceeds as follows:\medskip

{\sc Step 1.} Apply {\sc Approximating Genus Dense} to obtain an $r(\varepsilon)$-regular partition $\mathcal{P}$ into $K$ parts, where $r(\varepsilon)$ is the value of $\varepsilon^\prime$ in Theorem \ref{thm:genus}. 
Determine the quotient graph $H$, and construct a multigraph $H^*$ with multiplicity bounded by $m$ as in the proof of Theorem~\ref{thm:genus}, and find a triangle packing of $H^*$ which is close to its fractional packing as given in Theorem~\ref{thm: packing weighted graph}. Let $\mathfrak{T}^*$ be the collection of triangles we obtained.
We randomly partition the graph $G$ into $\mathsf{b}=O(K^2)$ bipartite graphs $\mathcal{B}_{ij}$ and $\mathsf{t}=O(K^3)$ tripartite graphs $\mathcal{T}_{ijk}$ and a small leftover graph $G_0$ with less than $\tfrac{\varepsilon}{10}(n/K)^2$ edges. In $G_0$ we put all edges in irregular parts of the regular partition and all edges between parts whose density is too small. For the remaining edges we determine the quotient graph $H$ and compute the value $\nu_{K_3}(H^*)$ as well as the family of triangles $\mathfrak{T}^*$ in $H^*$ and their balanced edge densities $\psi(T)/m$, $T\in\mathfrak{T}^*$. Then we partition the edges randomly in order to obtain, for each $T_{ijk}\in\mathfrak{T}$, a quasirandom tripartite graph $\mathcal{T}_{ijk}$ between parts $V_i$, $V_j$, $V_k$ with edge-density $\psi(T_{ijk})/m$. For each edge $ij\in E(H)$, if $b_{ij} = d_{ij} - \sum_{T\ni ij, T\in\mathcal{T}}\psi(T)/m > 0$, we put the remaining edges to form a quasirandom bipartite  subgraph $\mathcal{B}_{ij}$ between parts $V_i,V_j$ with edge-density $b_{ij}$. With high probability, for any $\tau>0$, at least $(1-r(\varepsilon)-\tau)\mathsf{b}$ of the resulting bipartite graphs are quasirandom and at least $(1-r(\varepsilon)-\tau)\mathsf{t}$ tripartite graphs are quasirandom. Note that $r(\varepsilon)$ appears here because of the irregular pairs.\medskip

{\sc Step 2.} Let $t_1=n^{\frac{2-\varepsilon}{4-\varepsilon}}$. For every $1\leq i<j<k\leq K$, we partition the edge-set of the graph $\mathcal{T}_{ijk}$ into $t_1$ parts uniformly at random. These $t_1$ sets of edges give us $t_1$ graphs $^1\mathcal{T}_{ijk},\dots,{^{t_1}\mathcal{T}_{ijk}}$. For any $\tau>0$, with probability at least $(1-\tau)t_1$, the graphs are still quasirandom. For every $x\in[t_1]$, let $^xD_{ijk}$ be the corresponding digraph of $^{x}\mathcal{T}_{ijk}$, and let $^x\H_{ijk}$ be the hypergraph whose vertices are the arcs of $^{x}\mathcal{T}_{ijk}$ and whose edges are the directed 3-cycles in $^{x}\mathcal{T}_{ijk}$.\medskip

{\sc Step 3.} Let $t_2=n^{\frac{4-\varepsilon}{6-\varepsilon}}$. For every $1\leq i<j\leq K$, we partition the edge-set of the graph $\mathcal{B}_{ij}$ into $t_2$ parts uniformly at random. These $t_2$ sets of edges give us $t_2$ graphs $^1\mathcal{B}_{ij},\dots,{^{t_2}\mathcal{B}_{ij}}$. For any $\tau>0$, with probability at least $(1-\tau)t_2$, the graphs are still quasirandom. For every $y\in[t_2]$, let $^yD_{ij}$ be the corresponding digraph of $^{y}\mathcal{B}_{ij}$, and let $^y\H_{ij}$ be the 4-uniform hypergraph, whose vertices are the arcs of $^{x}\mathcal{B}_{ij}$ and whose edges are the directed 4-cycles in $^{x}\mathcal{B}_{ij}$. Let $\mathsf{h}$ be the total number of hypergraphs.\medskip

{\sc Step 4.} Apply random greedy algorithm \cite{RoTh96} on $3$-uniform hypergraphs $^x\H_{ijk}$ for every $x\in[t_1]$ and every $1\leq i<j<k\leq K$, and on $4$-uniform hypergraphs $^y\H_{ij}$ for every $y\in[t_2]$ and every $1\leq i<j\leq K$ for finding near-perfect matchings. Note that the expected running time here is still quadratic even though we need to run the algorithm $\Theta(t_1+t_2)$ times. This is because the running time is linear in the number of vertices of the hypergraph, which means it is linear in the number of edges in the corresponding digraphs $^xD_{ijk}$ and $^yD_{ij}$.
With high probability, we have $|V(^x\H_{ijk})| = O(\tfrac{n^2}{t_1})$ and $|V(^y\H_{ij})| = O(\tfrac{n^2}{t_2})$. Then for every $\tau>0$, the algorithm will output a $\tau$-near perfect matching in at least $(1-r(\varepsilon)-\tau)\mathsf{h}$ hypergraphs with high probability. Let $\mathfrak{M}$ be the set of hyperedges (triangles and $4$-cycles) such that for every $e\in\mathfrak{M}$, $e$ is output by the algorithm as an element of a $\tau$-near perfect matching in a hypergraph $\H$.\medskip

{\sc Step 5.} For each hypergraph $\H$ we defined in {\sc Step 4}, consider $\H^{-1}$. Delete all the edges contained in $\mathfrak{M}$ (with inverse direction) from $\H^{-1}$. By \cite[Theorem 3.3]{genus}, the resulting hypergraph still satisfies Conditions (1)--(3) in Theorem \ref{thm:hypergraphmatching}. Apply random greedy algorithm again. For every $\tau>0$, the algorithm will output a $\tau$-near perfect matching in at least $(1-r(\varepsilon)-\tau)\mathsf{h}$ hypergraphs with high probability. We also put edges in these near perfect matchings into $\mathfrak{M}$.\medskip

{\sc Step 6.} Output the rotation $\Pi=\{\pi_v\mid v\in V(G)\}$ which is constructed as follows.
Every hyperedge $C\in \mathfrak{M}$ corresponds to a 3-cycle or a 4-cycle in $G$. Let $V_C$ and $E_C$ be the vertex and the edge-set of that cycle. Now define
$$ F(v) = \{ C\in \mathfrak{M} \mid v\in V_C \}.$$
The elements in $\mathfrak{M}$ determine which pairs of edges should be consecutive in the local rotation around $v$.
Each arc of the corresponding digraph $D$ of $G$ is in at most one of the cycles $C\in \mathfrak{M}$ since the digraphs $^xD_{ijk}$ and $^yD_{ij}$ are edge-disjoint. Thus the cycles $C\in \mathfrak{M}$ are arc disjoint cycles in $D\cup D^{-1}$.
For each blossom created by $\mathfrak{M}$, we remove one of the cycles from $\mathfrak{M}$. As proved in the paper, there is only a small number of blossoms all together with high probability. Obtaining a blossom-free subset, we apply Lemma \ref{lem:rotation} (following its proof).
By Theorem \ref{thm:genus}, we have $(1+\varepsilon)\g(G)\geq\g(G,\Pi)$.

\bibliographystyle{abbrv}
\bibliography{reference}

\end{document}